\pdfoutput=1
\RequirePackage[svgnames,x11names,hyperref]{xcolor}
\documentclass[graybox]{svmult}

\usepackage[utf8]{inputenc}

\usepackage{type1cm}        
\usepackage{makeidx}         
\usepackage{graphicx}        
\usepackage{multicol}        
\usepackage[bottom]{footmisc}

\usepackage{amsmath}
\usepackage{amssymb}
\usepackage{newtxtext}       %
\usepackage{newtxmath}       
\usepackage{mathtools}
\usepackage{enumitem}\setlist{leftmargin=*}

\definecolor{hrefcolor}{rgb}{0.0,0.4,0.7}
\definecolor{citecolor}{rgb}{0.0,0.35,0.2}
\definecolor{structure}{rgb}{0.09,0.09,0.44}
\definecolor{halfgray}{gray}{0.55}
\usepackage[
    colorlinks=true,
    linkcolor=citecolor,
    citecolor=citecolor,
    filecolor=hrefcolor,
    urlcolor=hrefcolor,
    pdfencoding=auto,
    hypertexnames=false,
]{hyperref}

\makeatletter
\def\appendix{
    \par
    \setcounter{section}{0}%
    \setcounter{subsection}{0}%
    \gdef\thesection{\@Alph\c@section}
}
\makeatother

\usepackage[nameinlink,capitalize]{cleveref}
\spnewtheorem{assumption}{Assumption}{\bfseries}{\rmfamily}
\crefname{figure}{Figure}{Figures}
\crefname{assumption}{Assumption}{Assumptions}
\crefformat{equation}{\upshape{(#2#1#3)}}
\crefrangeformat{equation}{\upshape{(#3#1#4)--(#5#2#6)}}
\crefmultiformat{equation}{\upshape{(#2#1#3)}}%
    { and~\upshape{(#2#1#3)}}{, \upshape{(#2#1#3)}}{ and~\upshape{(#2#1#3)}}
\crefformat{enumi}{\upshape{#2#1#3}}
\crefrangeformat{enumi}{\upshape{#3#1#4--#5#2#6}}
\crefmultiformat{enumi}{\upshape{#2#1#3}}%
    { and~\upshape{#2#1#3}}{, \upshape{#2#1#3}}{ and~\upshape{#2#1#3}}

\usepackage[textsize=footnotesize,color=DarkOrange!40]{todonotes}

\makeindex             

\newcommand{\field}[1]{\mathbb{#1}}
\newcommand{\N}{\mathbb{N}}

\newcommand{\R}{\field{R}}
\newcommand{\extR}{\overline \R}

\newcommand{\B}{B}
\newcommand{\norm}[1]{\|#1\|}

\newcommand{\abs}[1]{|#1|}

\newcommand{\inv}[1]{#1^{-1}}
\newcommand{\grad}{\nabla}
\newcommand{\freevar}{\,\boldsymbol\cdot\,}
\newcommand{\Union}\bigcup
\newcommand{\Isect}\bigcap
\newcommand{\union}\cup
\newcommand{\isect}\cap
\newcommand{\bigunion}\bigcup
\newcommand{\bigisect}\bigcap

\newcommand{\defeq}{:=}

\newcommand{\downto}{\searrow}
\newcommand{\upto}{\nearrow}
\newcommand{\subdiff}{\partial}

\DeclareMathOperator*{\argmin}{arg\,min}

\DeclareMathOperator{\dom}{dom}

\DeclareMathOperator{\divergence}{div}

\def \uminusSym{\setbox0=\hbox{$\cup$}\rlap{\hbox 
        to\wd0{\hss\raise0.5ex\hbox{$\scriptscriptstyle{-}$}\hss}}\box0}
    
\newcommand{\mathvar}[1]{\textup{#1}}

\newcommand{\iprod}[2]{\langle #1,#2\rangle}
\newcommand{\dualprod}[2]{\langle #1 | #2\rangle}

\newcommand{\weakto}{\mathrel{\rightharpoonup}}
\def \weaktostarsym{\setbox0=\hbox{$\rightharpoonup$}\rlap{\hbox 
        to\wd0{\hss\raise1ex\hbox{$\scriptscriptstyle{*\,}$}\hss}}\box0}
    \def \weaktostar    {\mathrel{\weaktostarsym}}

\newcommand{\setto}{\rightrightarrows}

\def\opt#1{\bar #1}
\def\realopt#1{\hat #1}

\def\alt#1{\tilde #1}

\def\this#1{#1^k}
\def\nexxt#1{#1^{k+1}}

\def\prev#1{#1^{k-1}}

\def\optu{{\opt{u}}}

\def\optx{{\opt{x}}}
\def\opty{{\opt{y}}}

\def\realoptu{{\realopt{u}}}

\def\realoptx{{\realopt{x}}}
\def\realopty{{\realopt{y}}}

\def\nextu{\nexxt{u}}

\def\nextx{\nexxt{x}}
\def\nexty{\nexxt{y}}

\def\thisu{\this{u}}

\def\thisx{\this{x}}
\def\thisy{\this{y}}

\def\prevu{\prev{u}}

\def\prevx{\prev{x}}
\def\prevy{\prev{y}}

\def\d{\,\mathrm{d}}

\let\phi\varphi

\let\epsilon\varepsilon

\def\linear{\mathbb{L}}

\def\Meas{\mathcal{M}}
\def\BVspace{\mathvar{BV}}

\def\GenGap{\mathcal{G}}

\DeclareMathOperator{\prox}{prox}

\def\lagrangian{\mathcal{L}}

\renewrobustcmd{\downto}{{{\mathchoice%
            {\rotatebox[origin=c]{-20}{$\to$}}
            {\rotatebox[origin=c]{-20}{$\to$}}
            {\rotatebox[origin=c]{-20}{\scalebox{0.75}{$\to$}}}
            {\rotatebox[origin=c]{-20}{\scalebox{0.6}{$\to$}}}
}}}

\renewrobustcmd{\upto}{{{\mathchoice%
            {\rotatebox[origin=c]{20}{$\to$}}
            {\rotatebox[origin=c]{20}{$\to$}}
            {\rotatebox[origin=c]{20}{\scalebox{0.75}{$\to$}}}
            {\rotatebox[origin=c]{20}{\scalebox{0.6}{$\to$}}}
}}}

\newcommand{\termnoindex}[1]{\emph{#1}}
\newcommand{\term}[2][]{\emph{#2}\if\relax\detokenize{#1}\relax\index{#2}\else\index{#1}\fi}
\newcommand{\indexalso}[1]{#1\index{#1}}

\newenvironment{alg}[1]{\par\noindent\begin{minipage}{\textwidth}\begin{programcode}{#1}}{\end{programcode}\end{minipage}\par\addvspace{\baselineskip}}
\newenvironment{fixedimportant}[1]{\par\noindent\begin{minipage}{\textwidth}\begin{important}{#1}}{\end{important}\end{minipage}\par\addvspace{\baselineskip}}

\begin{document}

\title*{First-order primal-dual methods for nonsmooth nonconvex optimisation}
\author{Tuomo Valkonen}
\institute{Tuomo Valkonen \at Center for Mathematical Modeling, Escuela Politécnica Nacional, Quito, Ecuador \emph{and} Department of Mathematics and Statistics, University of Helsinki, Finland; \email{tuomo.valkonen@iki.fi}}

\maketitle

\abstract{
    We provide an overview of primal-dual algorithms for nonsmooth and non-convex-concave saddle-point problems.
    This flows around a new analysis of such methods, using Bregman divergences to formulate simplified conditions for convergence.
}

\section{Introduction}
\label{sec:intro}

Interesting imaging problems can often be written in the general form \index{problem!min-max}\index{problem!saddle point}
\begin{equation}
    \label{eq:intro:genprob}
    \tag{S}
    \min_{x \in X}\max_{y \in Y}~ F(x) + K(x, y) - G_*(y),
\end{equation}
where $X$ and $Y$ are Banach spaces, $K \in C^1(X, Y)$, and $F: X \to \extR$ and $G_*: Y \to \extR$ are convex, proper, lower semicontinuous functions with $G_*$ the \hyperref[sec:glossary]{preconjugate }of some $G: Y^* \to \extR$, meaning $G=(G_*)^*$. The functions $F$ and $G_*$ may be nonsmooth.
In this chapter, we provide an overview of proximal-type primal-dual algorithms for this class of problems together with a simplified analysis, based on Bregman divergences.

\begin{fixedimportant}{Notation, conventions, and basic convex analysis}%
    As is standard in optimisation, all vector/Banach/Hilbert spaces in this chapter are over the real field without it being explicitly mentioned.
    For basic definitions of convex analysis, such as the (pre)conjugate and the subdifferential, see the \hyperref[sec:glossary]{glossary} at the end of the chapter or textbooks such as \cite{hiriarturruty2004fundamentals,rockafellar-convex-analysis,clasonvalkonen2020nonsmooth,ekeland1999convex}.
\end{fixedimportant}

A common instance of \eqref{eq:intro:genprob} is when $K(x, y)=\dualprod{Ax}{y}$ for a linear operator $A \in \linear(X; Y^*)$ with $\dualprod{\freevar}{\freevar}: Y^* \times Y \to \R$ denoting the dual product. Then \eqref{eq:intro:genprob} arises from writing $G$ in terms of its (pre)conjugate $G_*$ in \index{problem!primal}
\begin{equation}
    \label{eq:intro:fgprob}
    \min_{x \in X}~ F(x) + G(Ax).
\end{equation}
We now discuss sample imaging and inverse problems of the types \eqref{eq:intro:genprob} and \eqref{eq:intro:fgprob}, and then outline our approach to solving them in the rest of the chapter.

\subsection{Sample problems}

\index{inverse problem|(}
Optimisation problems of the type \eqref{eq:intro:fgprob} can effectively model linear \term[inverse problem!linear]{inverse problems}; typically one would attempt to minimise the sum of a data-term and a regulariser,
\begin{equation}
    \label{eq:intro:invprob}
    \min_{x \in X}~ \Phi(z-Tx) + G(Ax),
\end{equation}
where
\begin{itemize}[label={--}]
    \item $T: \in \linear(X; \R^n)$ is a forward operator, mapping our unknown $x$ into a finite number of measurements. .
    \item $\Phi$ models noise $\nu$ in the data $z \in \R^n$; for normal-distributed noise, $\Phi(z)=\frac{1}{2}\norm{z}^2$;
    \item $G \circ A$ is a typically nonsmooth regularisation term that models our prior assumptions on what a good solution to the ill-posed problem $z=Tx+\nu$ should be; in imaging, what ``looks good''. 
\end{itemize}
For conventional total variation regularisation on a domain $\Omega \subset \R^m$ one would take $G(y^*)=\alpha \norm{y^*}_{\Meas(\Omega; \R^m)}$ the Radon norm of the measure $y^* \in \Meas(\Omega; \R^m)$ weighted by the regularisation parameter $\alpha>0$, and $A=D \in \linear(\BVspace(\Omega); \Meas(\Omega; \R^m))$ the \hyperref[sec:glossary]{distributional derivative} \cite{ambrosio2000fbv}.
Simple examples of a \emph{linear} forward operator $T$ include:
\begin{itemize}[label={--}]
    \item the identity for denoising \cite{Rud1992},
    \item a convolution operation for deblurring or deconvolution \cite{vogel1998fast},
    \item a subsampling operator for inpainting \cite{shen2002mathematical},
    \item \index{transform!Fourier}\index{magnetic resonance imaging|see {MRI}}\index{MRI} the Fourier transform for magnetic resonance imaging (\indexalso{MRI}) \cite{nishimura1996principles,lustig2007sparse}, and
    \item \index{transform!Radon} the Radon transform for \index{tomography!computational}computational (\indexalso{CT}) or \index{tomography!positron emission} positron emission tomography (\indexalso{PET}) \cite{ollinger1997pet}.
\end{itemize}
The last two examples would frequently be combined with subsampling for reconstruction from limited data.

In many important problems $T$ is, however, nonlinear:
\begin{itemize}[label={--}]
    \item \index{MRI!velocity-encoded} a pointwise application of $(r,\phi) \mapsto  re^{-\mathbf{i}\phi}$ for phase and amplitude reconstruction for velocity-encoded magnetic resonance imaging \cite{tuomov-nlpdhgm},
    \item \index{MRI!diffusion}\index{DTI} a pointwise application of $u \mapsto s_0 - s e^{-\iprod{u}{b}}$ to model the Stejskal--Tanner equation in diffusion tensor imaging \cite{tuomov-nlpdhgm,kingsley2006introduction}, or
    \item \index{tomography!electrical impedance}\index{tomography!acoustic}\index{tomography!optical} the solution operator of nonlinear partial differential equation (PDE) for several forms of tomography from magnetic and electric to acoustic and optical \cite{nishimura1996principles,ollinger1997pet,arridge2011optical,kuchment2011mathematics,hunt2014weighing,trucu2009bioheat,uhlmann2009eit,lipponen2011nonstationary}.
\end{itemize}
In the last example, the PDE governs the physics of measurement, typically relating boundary measurements and excitations to interior data. The methods we study in this chapter are applied to electrical impedance tomography in \cite{jauhiainen2019gaussnewton,tuomov-nlpdhgm-block}.

How to fit a \index{inverse problem!nonlinear}nonlinear forward operator $T$ into the framework \cref{eq:intro:genprob} that requires both $F$ and $G_*$ to be convex? If the noise model $\Phi: \R^n \to \extR$ is convex, proper, and lower semicontinuous, we can write \eqref{eq:intro:invprob} using the Fenchel conjugate $\Phi^*$ and $K_{TA}(x, (y_1, y_2)) \defeq \dualprod{z-T(x)}{y_1}+\dualprod{Ax}{y_2}$ as
\begin{equation}
    \label{eq:intro:twoblock}
    \min_{x \in X} \max_{(y_1, y_2) \in \R^n \times Y}~ K_{TA}(x,(y_1,y_2)) - \Phi^*(y_1) - G_*(y_2).
\end{equation}

This is of the form \cref{eq:intro:genprob} for the functions $\alt F \equiv 0$ and $\alt G_*(y_1, y_2) \defeq \Phi^*(y_1) - G_*(y_2)$.
Even for linear $T$, although \eqref{eq:intro:invprob} is readily of the form \eqref{eq:intro:fgprob} and hence \eqref{eq:intro:genprob}, this reformulation may allow expressing \eqref{eq:intro:invprob} in the form \eqref{eq:intro:genprob} with both $\alt F$ and $\alt G_*$ “prox-simple”.
We will make this concept, important for the effective realisation of algorithms, more precise in \cref{sec:pdbs}.

\index{inverse problem|)}

Finally, fully general $K$ in \eqref{eq:intro:genprob} was shown in \cite{tuomov-nlpdhgm-general} to be useful for highly nonsmooth and nonconvex problems, such as the  \index{segmentation!Potts} \cite{geman1984stochastic}. Indeed, the “0-function”
\[
    \abs{t}_0 \defeq
    \begin{cases}
        0, & t=0, \\
        1, & t \ne 0,
    \end{cases}
\]
can be written 
\[
     \abs{t}_0 = \sup_{s \in \R} \rho(st)
     \quad\text{for}\quad \rho(t)=2t-t^2.
\]
For the (anisotropic) Potts model this is applied pixelwise on a discretised image gradient computed for an $n_1 \times n_2$ image by $\grad_h: \R^{n_1n_2} \to \R^{2 \times n_1n_2}$ \cite{tuomov-nlpdhgm-general}:
\begin{equation}
    \label{eq:intro:potts}
    \min_{x \in \R^{n_1n_2}} \max_{y \in \R^{2 \times n_1n_2}} \frac{1}{2}\norm{b-x}_2^2 + \sum_{i=1}^{n_1}\sum_{j=1}^{n_2} \rho(\iprod{[\grad_h x]_{ij}}{y_{ij}}),
\end{equation}
where $b \in \R^{n_1n_2}$ is the image to be segmented.

\subsection{Outline}

We introduce in \cref{sec:pdbs} methods for \eqref{eq:intro:genprob} inspired by the \term[primal-dual!splitting!proximal]{primal-dual proximal splitting} (\indexalso{PDPS}) of \cite{chambolle2010first,pock2009mumford} for bilinear $K$, commonly known as the \term{Chambolle--Pock method}. 
We work in Banach spaces, as was done in \cite{hohage2014generalization}. To be able to define proximal-type methods in Banach spaces, in \cref{sec:bregman}, we introduce and recall the crucial properties of so-called \term[Bregman divergence]{Bregman divergences}.

Our main reason for working with Bregman divergences is, however, not the generality of Banach spaces. Rather, they provide a powerful proof tool to deal with the general $K$ in \eqref{eq:intro:genprob}. This approach allows us in \cref{sec:convergence} to significantly simplify and better explain the original convergence proofs and conditions of \cite{chambolle2010first,tuomov-nlpdhgm,tuomov-nlpdhgm-redo,tuomov-nlpdhgm-general,tuomov-nlpdhgm-block}. Without additional effort, they also allow us to present block-adapted methods like those in \cite{tuomov-cpaccel,tuomov-blockcp,tuomov-nlpdhgm-block}.

\hypertarget{target:roadmap}{Our overall approach} and the internal organisation of \cref{sec:convergence} centres around the following three main ingredients of the convergence proof:
\begin{enumerate}[label=(\roman*)]
    \item A \textbf{three-point identity}, satisfied by all Bregman divergences (shown in \cref{sec:bregman} and employed in \cref{sec:convergence:fundamental}),
    \item \textbf{(Semi-)ellipticity} of the algorithm-defining Bregman divergences (concept defined in \cref{sec:bregman}, specific Bregman divergence in \cref{sec:pdbs}, and its ellipticity verified in \cref{sec:convergence:ellipticity,sec:convergence:ellipticity:block} through several examples), and
    \item A \textbf{non-smooth second-order growth} condition around a solution of \eqref{eq:intro:genprob} (treated in \cref{sec:convergence:second,sec:convergence:second:block}).
\end{enumerate}
With these basic ingredients, we then prove convergence in \cref{sec:convergence:iterate,sec:convergence:gaps}.
In the present overview, with focus on key concepts and aiming to avoid technical complications, we only cover, weak, strong, and linear convergence of iterates, and the convergence of gap functionals when $K$ is convex-concave.

In \cref{sec:inertial} we improve the basic method by adding dependencies to earlier iterates, a form of inertia. This is needed to develop an an effective algorithm for $K$ not affine in $y$, including the aforementioned formulation of the Potts segmentation model.
We finish in \cref{sec:further} with pointers to alternative methods and further extensions.

\section{Bregman divergences}
\label{sec:bregman}

The norm and inner product in a (real) Hilbert space $X$ satisfy the \indexalso{three-point identity}
\begin{equation}
    \label{eq:bregman:hilbert-three-point-identity}
    \iprod{x-y}{x-z}_X
    = \frac{1}{2}\norm{x-y}_X^2
    - \frac{1}{2}\norm{y-z}_X^2
    + \frac{1}{2}\norm{x-z}_X^2 \quad(x,y,z\in X).
\end{equation}
This is crucial for convergence proofs of optimisation methods \cite{tuomov-proxtest}, so we would like to have something similar in Banach spaces---or other more general spaces.
Towards this end, we let $J: X \to \R$ be a Gâteaux-differentiable function.\footnote{The differentiability assumption is for notational and presentational simplicity; otherwise we would need to write the Bregman divergence as $B_J^p(z, x) \defeq J(z)-J(x)-\dualprod{p}{z-x}_X$ for some subdifferential $p$ of $J$, and define explicit updates of this subdifferential in algorithms.}
Then one can define the asymmetric \term{Bregman divergence}
\begin{equation}
    \label{eq:bregman:def}
    B_J(z, x) \defeq J(z)-J(x)-\dualprod{D J(x)}{z-x}_X
    \quad (x, z \in X).
\end{equation}
This function is non-negative \emph{if and only if}\footnote{For the entirely algebraic proof of the ``only if'', see \cite[Theorem 4.1.1]{hiriarturruty2004fundamentals}.} the \term{generating function} $J$ is convex; it is not in general a true distance, as it can happen that $B_J(x, z)=0$ although $x=z$.


Writing $D_1$ for the Gâteaux derivative with respect to the first parameter, we have
\begin{equation}
    \label{eq:bregman:d}
    D_1 B_J(x, z) = DJ(z)-DJ(x).
\end{equation}
Moreover, the Bregman divergence satisfies for any $\optx \in X$ the \term{three-point identity}
\begin{equation}
    \label{eq:bregman:three-point-identity}
    \begin{aligned}[t]
    \dualprod{D_1 B_J(x, z)}{x-\optx}_X
    &
    =
    \dualprod{D J(x)-D J(z)}{x-\optx}_X
    \\
    &
    =
    B_J(\optx, x)  - B_J(\optx, z) + B_J(x,z).
    \end{aligned}
\end{equation}
Indeed, writing the right-hand side out, we have
\[
    \begin{split}
    B_J(\optx, x)  - B_J(\optx, z) + B_J(x,z)
    &
    =
    [J(\optx)-J(x)-\dualprod{D J(x)}{\optx-x}_X]
    \\
    \MoveEqLeft[-1]
    -
    [J(\realoptx)-J(z)-\dualprod{D J(z)}{\realoptx-z}_X]
    \\
    \MoveEqLeft[-1]
    +
    [J(x)-J(z)-\dualprod{D J(z)}{x-z}_X],
    \end{split}
\]
which immediately gives the three-point identity.

\begin{example}
    In a Hilbert space $X$, the \term[generating function!standard]{standard} \term{generating function} $J= N_X \defeq \frac{1}{2}\norm{\freevar}_X^2$ yields $B_J(z, x)=\frac{1}{2}\norm{z-x}_X^2$, so \eqref{eq:bregman:three-point-identity} recovers \eqref{eq:bregman:hilbert-three-point-identity}.
\end{example}

We will frequently require $B_J$ to be \term{non-negative} or \term{semi-elliptic} ($\gamma=0$) or \term{elliptic} ($\gamma>0$) within some $\Omega \subset X$. These notions mean that
\begin{equation}
    \label{eq:bregman:ellipticity}
    B_J(z, x) \ge \frac{\gamma}{2}\norm{z-x}_{X}^2
    \quad (x, z \in \Omega).
\end{equation}
Equivalently, this defines $J$ to be \term[subdifferentiable!strongly]{($\gamma$-strongly) subdifferentiable} within $\Omega$. When $\Omega=X$, we simply call $B_J$ \mbox{(semi-)}elliptic and $J$ ($\gamma$-strongly) subdifferentiable.\footnote{In Banach spaces strong subdifferentiability is implied by strong convexity, as defined without subdifferentials. In Hilbert spaces the two properties are equivalent.}

We will in \cref{sec:inertial} also need a Cauchy inequality for Bregman divergences. We base this on strong subdifferentiability and the smoothness property \eqref{eq:smoothness-grad-smoothness} in the next lemma. The latter holding with $\Omega=X$ implies that $DJ$ is $L$-Lipschitz, and in Hilbert spaces is equivalent to this property; see \cite[Theorem 18.15]{bauschke2017convex} or \cite[Appendix C]{tuomov-proxtest}.

\begin{lemma}
    \label{lemma:bregman:cauchy}
    Suppose $J: X \to \R$ is Gâteaux-differentiable and $\gamma$-strongly subdifferentiable within $\Omega$, and satisfies for some $L>0$ the \term{subdifferential smoothness}
    \begin{equation}
        \label{eq:smoothness-grad-smoothness}
        \frac{1}{2L}\norm{D J(x)-D J(y)}_{X^*}^2
        \le J(x) - J(y) - \dualprod{D J(y)}{x-y}
        \quad (x, y \in \Omega).
    \end{equation}
    Then, for any $\alpha>0$, 
    \begin{equation*}
        \abs{\dualprod{D_1 B_J(x, y)}{z-x}}
        \le \frac{L}{\alpha} B_J(x, y) + \frac{\alpha}{\gamma} B_J(z, x)
        \quad (x, y, z \in \Omega).
    \end{equation*}
\end{lemma}

\begin{proof}
    By Cauchy's inequality and \eqref{eq:bregman:d},
    \[
        \abs{\dualprod{D_1 B_J(x, y)}{z-x}}
        \le
        \frac{1}{2\alpha}\norm{DJ(x)-DJ(y)}_{X^*}^2
        +
        \frac{\alpha}{2}\norm{z-x}_{X}^2.
    \]
    By the strong convexity, $\frac{\gamma}{2}\norm{z-x}_{X}^2  \le B_J(z,x)$, and by the smoothness property \eqref{eq:smoothness-grad-smoothness},
    $
        \frac{1}{2L}\norm{DJ(x)-DJ(y)}_{X^*}^2
        \le B_J(x, y)
    $.
    Together these estimates yield the claim.
\end{proof}

\section{Primal-dual proximal splitting}
\label{sec:pdbs}

We now formulate a basic version of our primal-dual method. Later in \cref{sec:inertial} we improve the algorithm to be more effective when $K$ is not affine in $y$.

\begin{fixedimportant}{Notation}
    Throughout the manuscript, we combine the primal and dual variables $x$ and $y$ into variables involving the letter $u$:
    \[
        u=(x, y),
        \quad
        \thisu=(\thisx,\thisy),
        \quad
        \realoptu=(\realoptx, \realopty),
        \quad\text{etc.}
    \]
\end{fixedimportant}

\medskip

\subsection{Optimality conditions and proximal points}

We define the \term{Lagrangian} as
\[
    \lagrangian(x, y) \defeq F(x) + K(x, y) - G_*(y).
\]
A \term{saddle point} $\realoptu=(\realoptx,\realopty)$ of the problem \eqref{eq:intro:genprob} satisfies, by definition
\[
    \lagrangian(\realoptx, y) \le \lagrangian(\realoptx,\realopty) \le \lagrangian(x, \realopty)
    \quad
    \text{for all } u=(x,y) \in X \times Y.
\]
Writing $D_x K$ and $D_y K$ for the Gâteaux derivatives of $K$ with respect to the two variables, if $K$ is convex-concave, basic results in convex analysis \cite{ekeland1999convex,bauschke2017convex} show that
\begin{equation}
    \label{eq:convergence:oc}
    -D_x K(\realoptx, \realopty) \in \subdiff F(\realoptx)
    \quad\text{and}\quad
    D_y K(\realoptx, \realopty) \in \subdiff G_*(\realopty)
\end{equation}
is necessary and sufficient for $\realoptu$ to be saddle point.
If $K$ is $C^1$, the theory of generalised subdifferentials of Clarke \cite{clarke1990optimization} still indicates\footnote{The Fermat-rule $0 \in \subdiff_C [F+K(\freevar, \realopty)](\realoptx)$ holds. Since $F$ is convex and $K(\freevar, \realopty)$ is $C^1$, $\realoptx$ is a regular point of both, so also the subdifferential sum rule holds. We argue $G_*+K(\realopty, \freevar)$ similarly.} the necessity of \eqref{eq:convergence:oc}.

We can alternatively write \eqref{eq:convergence:oc} as
\begin{equation}
    \label{eq:convergence:h}
    0 \in H(\realoptu)
    \defeq
    \begin{pmatrix}
        \subdiff F(\realoptx) + D_x K(\realoptx, \realopty) \\
        \subdiff G_*(\realopty) - D_y K(\realoptx, \realopty)
    \end{pmatrix}.
\end{equation}
If $X$ and $Y$ were Hilbert spaces, we could in principle use the classical \term{proximal point method} \cite{minty-monotone,rockafellar1976monotone} to solve \eqref{eq:convergence:h}: given step length parameters $\tau_k>0$, iteratively solve $\nextu$ from
\begin{equation}
    \label{eq:bregpdps:basicprox}
    0 \in H(\nextu) + \inv\tau_k(\nextu-\thisu).
\end{equation}
If $K$ were bilinear, $H$ would be a so-called monotone operator and convergence of iterates would follow from \cite{rockafellar1976monotone}.
In practise the steps of the method are too expensive to realise as the primal and dual iterates $\nextx$ and $\nexty$ are coupled: generally, one cannot solve one before the other.

Fortunately, the iterates can be decoupled by introducing a \term{preconditioner} that switches $D_x K(\nextx, \nexty)$ on the first line of $H(\nextu)$ to $D_x K(\thisx, \thisy)$. This gives rise to the \term[primal-dual!splitting!proximal]{primal-dual proximal splitting} (\indexalso{PDPS}), introduced in \cite{chambolle2010first,pock2009mumford} for bilinear $K(x,y)=\dualprod{Ax}{y}$. That the PDPS is actually a preconditioned proximal point method was first observed in \cite{he2012convergence}. In the following, we describe its extension from \cite{tuomov-nlpdhgm,tuomov-nlpdhgm-redo,tuomov-nlpdhgm-general} to general $K$ and the general problem \eqref{eq:intro:genprob}.
To simplify the proofs and concepts in them, we work with Bregman divergences, at no cost in Banach spaces.

\subsection{Algorithm formulation}

Given Gâteaux-differentiable functions $J_X: X \to \extR$ and $J_Y: Y \to \extR$ with the corresponding Bregman divergences $B_X \defeq B_{J_X}$ and $B_Y \defeq B_{J_Y}$, we define
\begin{equation}
    \label{eq:bregpdps:j0}
    J^0(x, y)
    \defeq
    J_X(x) + J_Y(y) - K(x, y).
\end{equation}
Introducing the short-hand notation $B^0 \defeq B_{J^0}$, we propose to solve \eqref{eq:convergence:h} through the iterative solution of
\begin{equation}
    \label{eq:bregpdps:alg}
    0 \in H(\nextu) + D_1 B^0(\nextu, \thisu)
\end{equation}
for $\nextu$.
Inserting \cref{eq:convergence:h} and \eqref{eq:bregman:d} for $J=J^0$ as defined in \eqref{eq:bregpdps:j0},
we expand and rearrange this implicitly defined method as:

\begin{alg}{Primal-dual Bregman-proximal splitting (PDBS)}%
    \index{PDBS}%
    \index{primal-dual!splitting!Bregman-proximal}%
    Iteratively over $k \in\N$, solve for $\nextx$ and $\nexty$:
    \begin{equation}
        \label{eq:bregpdps:alg:expanded}
        \!\!\!\!\!
            \begin{aligned}
                DJ_X(\thisx)-D_x K(\thisx, \thisy) & \in DJ_X(\nextx) + \subdiff F(\nextx)
                \quad\text{and}
                \\
                DJ_Y(\thisy)-D_y K(\thisx, \thisy) & \in DJ_Y(\nexty) + \subdiff G_*(\nexty)-2D_y K(\nextx, \nexty).
            \end{aligned}
    \end{equation}
\end{alg}

We readily obtain $\nextx$ if the inverse of $DJ_X+\tau \subdiff F$ has an analytical closed-form expression. In this case we say that $F$ is \term{prox-simple} with respect to $J_X$.
For $\nexty$, the same is true if $K$ is affine in $y$ and $G_*$ is prox-simple with respect to $J_Y$. If, however, $K$ is not affine in $y$, it is practically unlikely that $\subdiff G_*-2D_y K(\nextx,\freevar)$ would be prox-simple. We will therefore improve the method for general $K$ in \cref{sec:inertial}, after first studying fundamental ideas behind convergence proofs in the following \cref{sec:convergence}.

If $X$ and $Y$ are Hilbert spaces with $J_X=\inv\tau N_X$ and $J_Y=\inv\sigma N_Y$ the standard generating functions divided by some step length parameters $\tau,\sigma>0$, \eqref{eq:bregpdps:alg:expanded} becomes
\begin{alg}{Primal--dual proximal splitting (PDPS)}
    \index{PDPS}%
    \index{primal-dual!splitting!proximal}%
    Iterate over $k \in \N$:
    \begin{equation}
        \label{eq:bregpdps:alg:hilbert}
            \begin{aligned}
                \nextx & \defeq \prox_{\tau F}(\thisx - \tau \grad_x K(\thisx, \thisy)),
                \\
                \nexty & \defeq \prox_{\sigma[G_*-2K(\nextx, \freevar)]}(\thisy - \sigma\grad_y K(\thisx, \thisy)).
            \end{aligned}
    \end{equation}
\end{alg}%
\noindent
The \term{proximal map} is defined as
\[
    \prox_{\tau F}(x) \defeq \inv{(I+\tau \subdiff F)}(x) = \argmin_{\alt x \in X}\left( \tau F(\alt x)+\frac{1}{2}\norm{\alt x-x}_X^2\right).
\]
When this map has an analytical closed-form expression, we say that $F$ is \term{prox-simple} (without reference to $J_X$).
In finite dimensions, several worked out proximal maps may be found online \cite{chierchia2019proximity} or in the book \cite{beck2017firstorder}.
Some extend directly to Hilbert spaces or by superposition to $L^2$.

\begin{remark}
    \label{rem:bregpdps:nonlin}
    For $K$ affine in $y$, i.e., $K(x,y)=\dualprod{A(x)}{y}$ for some differentiable $A: X \to Y^*$, the dual update of \eqref{eq:bregpdps:alg:hilbert} reduces to
    \[
        \begin{aligned}[t]
        \nexty &
        = \prox_{\sigma G_*}(\thisy + \sigma[2\grad_y K(\nextx,\thisy)-\grad_y K(\thisx, \thisy)])
        \\
        &
        =
        \prox_{\sigma G_*}(\thisy + \sigma[2\grad A(\nextx)-\grad A(\thisx)]).
        \end{aligned}
    \]
    This corresponds to the “linearised” variant of the NL-PDPS of \cite{tuomov-nlpdhgm}. The “exact” variant, studied in further detail in \cite{tuomov-nlpdhgm-redo}, updates
    \[
        \nexty \defeq \prox_{\sigma G_*}(\thisy + \sigma \grad_y K(2\nextx-\thisx, \thisy)).
    \]
    If $K$ is bilinear the two variants are the exactly same PDPS of \cite{chambolle2010first}.
    For $K$ not affine in $y$, the method is neither the generalised PDPS of \cite{tuomov-nlpdhgm-general} nor the version for convex-concave $K$ from \cite{hamedani2018primal}.
\end{remark}

\subsection{Block-adaptation}
\label{sec:bregpdps:block}

\index{primal-dual!splitting!block-adapted|(}%
\index{PDPS!block-adapted|(}%

We now derive a version of the PDBS \eqref{eq:bregpdps:alg:expanded} adapted to the structure of
\[
    F(x)=\sum_{j=1}^m F_j(x_j)\quad\text{and}\quad
    G_*(y)=\sum_{\ell=1}^n G_{\ell*}(y_\ell),
\]
where $x=(x_1,\ldots,x_m)$ and $y=(y_1,\ldots,y_n)$ in the (for simplicity) Hilbert spaces $X=\prod_{j=1}^m X_j$ and $Y=\prod_{\ell=1}^n Y_k$, and $F_j: X_j \to \extR$ and $G_{\ell*}: Y_\ell \to \extR$ are convex, proper, and lower semicontinuous.

For some “blockwise” step length parameters $\tau_j,\sigma_\ell>0$ we take
\[
    J_X(x)=\sum_{j=1}^m \inv\tau_j N_{X_j}(x_j)
    \quad\text{and}\quad
    J_Y(y)=\sum_{\ell=1}^n \inv\sigma_\ell N_{Y_\ell}(y_\ell)
\]
If $K$ is now affine in $y$, observing \cref{rem:bregpdps:nonlin}, \eqref{eq:bregpdps:alg:expanded} readily transforms into:
\begin{alg}{Block-adapted PDPS for $K$ affine in $y$}%
    Iteratively over $k \in \N$, for all $j=1,\ldots,m$ and $\ell=1,\ldots,n$, update:
    \begin{equation}
        \label{eq:bregpdps:alg:block}
            \begin{aligned}
                \nextx_j & \defeq \prox_{\tau_j F_j}(\thisx_j-\tau_j \grad_{x_j} K(\thisx, \thisy)),
                \\
                \nexty_\ell & \defeq \prox_{\sigma_\ell G_{\ell*}}(\thisy_\ell +  \sigma_\ell[2\grad_{y_\ell} K(\nextx, \thisy) -\grad_{y_\ell}  K(\thisx, \thisy)]).
            \end{aligned}
    \end{equation}
\end{alg}

The idea is that the blockwise step length parameters adapt the algorithm to the structure of the problem. We will return their choices in the examples of \cref{sec:convergence:ellipticity:block}.

\begin{fixedimportant}{Performance gains}%
    Correct adaptation of the blockwise step length parameters to the specific problem structure can yield significant performance gains compared to not exploiting the block structure \cite{pock2011iccv,jauhiainen2019gaussnewton,tuomov-nlpdhgm-block}.
\end{fixedimportant}

\begin{remark}
    For bilinear $K$, \eqref{eq:bregpdps:alg:block} is the “diagonally preconditioned” method of \cite{pock2011iccv}, or an unaccelerated non-stochastic variant of the methods in \cite{tuomov-blockcp}.
    For $K$ affine in $y$, \eqref{eq:bregpdps:alg:block} differs from the methods in \cite{tuomov-nlpdhgm-block} by placing the over-relaxation in the dual step outside $K$, compare \cref{rem:bregpdps:nonlin}.
\end{remark}

Recall the saddle-point formulation \eqref{eq:intro:twoblock} for \index{inverse problem}inverse problems with nonlinear forward operators. We can now adapt step lengths to the constituent dual blocks:

\begin{example}
    \label{ex:bregpdps:twoblock}
    Let $A_1 \in C^1(X; Y_1^*)$ and $A_2 \in \linear(X; Y_2^*)$, and suppose the convex functions $G_1: Y_1^* \to \extR$ and $G_2: Y_2^* \to \extR$ have the preconjugates $G_{1*}$ and $G_{2*}$.
    Then we can write the problem
    \[
        \min_{x \in X}~G_1(A_1(x))+G_2(A_2x)+F(x).
    \]
    in the form \eqref{eq:intro:genprob} with $G_*(y_1, y_2)=G_{1*}(y_1)+G_{2*}(y_2)$ and $K(x,y)=\dualprod{A_1(x)}{y_1}+\dualprod{A_2x}{y_2}$. The algorithm \eqref{eq:bregpdps:alg:block} specialises as
    \begin{equation*}
            \begin{aligned}
                \nextx & \defeq \prox_{\tau F}(\thisx-\tau[\grad A_1(\thisx)^*y_1+A_2^*y_2]),
                \\
                \nexty_1 & \defeq \prox_{\sigma_1 G_{1*}}(\thisy_1 +  \sigma_1[2A_1(\nextx)-A_1(\thisx)]),
                \\
                \nexty_2 & \defeq \prox_{\sigma_2 G_{2*}}(\thisy_2 +  \sigma_2[A_2(2\nextx-\thisx)])
            \end{aligned}
    \end{equation*}
    for some step length parameters $\tau,\sigma_1,\sigma_2>0$.
    We return to their choices and the local neighbourhood of convergence in \cref{ex:bregpdps:twoblock:steps,ex:bregpdps:twoblock:nbd} after developing the necessary convergence theory.
    
\end{example}

\index{primal-dual!splitting!block-adapted|)}%
\index{PDPS!block-adapted|)}%

\section{Convergence theory}
\label{sec:convergence}

We now seek to understand when the basic version \eqref{eq:bregpdps:alg} of the PDBS convergences.
The organisation of this section centres around the \hyperlink{target:roadmap}{three main ingredients} of the convergence proof, as discussed in the Introduction:
\begin{enumerate}[label=(\roman*)]
    \item the three-point identity \eqref{eq:bregman:three-point-identity} employed in the general-purpose estimate of \cref{sec:convergence:fundamental},
    \item\label{item:convergence:ellipticity}
    the \mbox{(semi-)}ellipticity of the algorithm-generating Bregman divergences $B_{J_0}$ for $J^0$ as in \eqref{eq:bregpdps:j0}, verified for several examples in \cref{sec:convergence:ellipticity,sec:convergence:ellipticity:block}, and
    \item\label{item:convergence:growth}
    a second-order growth condition on \eqref{eq:intro:genprob}, verified for several examples in \cref{sec:convergence:second,sec:convergence:second:block}.
\end{enumerate}
With these basic ingredients, we then prove various convergence results in \cref{sec:convergence:iterate,sec:convergence:gaps}.
The usefulness of both \cref{item:convergence:ellipticity,item:convergence:growth} will become apparent from the fundamental estimates and examples of the next \cref{sec:convergence:fundamental}.

\subsection{A fundamental estimate}
\label{sec:convergence:fundamental}

We start with a simple estimate applicable to general methods of the form
\begin{equation}
    \label{eq:abstractbreg:pp}
    \tag{BP}
    0 \in H(\nextu) + D_1 B (\nextu,\thisu)
\end{equation}
for some set-valued $H: U \setto U^*$ and a Bregman divergence $B \defeq B_J$ generated by some Gâteaux-differentiable $J: U \to \R$.
We analyse \eqref{eq:abstractbreg:pp} following the “\indexalso{testing}” ideas introduced in \cite{tuomov-proxtest}, extending them to the Bregman--Banach space setting, however in a simplified constant-metric setting that cannot model accelerated methods.
The \term[gap!generic]{generic gap functional} $\GenGap(\nextu, \optu)$ in the next result models any function value differences available from $H$.
Its non-negativity will provide the basis for the aforementioned second-order growth conditions of \cref{sec:convergence:second,sec:convergence:second:block}.
We provide an example and interpretation after the theorem.

\begin{theorem}
    \label{thm:abstractbreg}
    On a Banach space $U$, let $H: U \setto U^*$, and let $B \defeq B_J$ be generated by a Gâteaux-differentiable $J: U \to \R$.
    Suppose \eqref{eq:abstractbreg:pp} is solvable for $\{\nextu\}_{k \in \N}$ given an initial iterate $u^0 \in U$. Let $N \ge 1$. If for all $k=0,\ldots,N-1$, for some $\optu \in U$ and $\GenGap(\nextu, \optu) \in \R$ the \term{fundamental condition}
    \begin{equation}
        \label{eq:abstractbreg:condition}
        \tag{C}
        \dualprod{\nexxt{h}}{\nextu-\optu}
        \ge
        \GenGap(\nextu, \optu)
        \quad (\nexxt h \in H(\nextu))
    \end{equation}
    holds, then so do the \term[Féjer-monotonicity]{quantitative $\Delta$-Féjer monotonicity}
    \begin{equation}
        \label{eq:abstractbreg:fejer}
        \tag{F}
        B(\optu, \nextu)
        +B(\nextu,\thisu)
        +\GenGap(\nextu, \optu)
        \le
        B(\optu, \thisu)
    \end{equation}
    and the \term{descent inequality}
    \begin{equation}
        \label{eq:abstractbreg:result}
        \tag{D}
        B(\optu, u^N)
        +\sum_{k=0}^{N-1} B(\nextu,\thisu)
        +\sum_{k=0}^{N-1} \GenGap(\nextu, \optu)
        \le
        B(\optu, u^0).
    \end{equation}
\end{theorem}

\begin{proof}
    We can write \eqref{eq:abstractbreg:pp} as
    \begin{equation}
        \label{eq:abstractbreg:pp-expand}
        0=\nexxt h + D_1 B(\nextu, \thisu)
        \quad\text{for some}\quad
        \nexxt h \in H(\nextu).
    \end{equation}
    Testing \eqref{eq:abstractbreg:pp-expand} by applying $\dualprod{\freevar}{\nextu-\optu}$ we obtain
    \begin{equation*}
        0 = \dualprod{\nexxt{h} + D_1 B(\nextu, \thisu)}{\nextu-\optu}.
    \end{equation*}
    We use the three-point identity \eqref{eq:bregman:three-point-identity} to transform this into
    \[
        B(\optu, \thisu) = \dualprod{\nexxt{h}}{\nextu-\optu}
        +B(\optu, \nextu)
        +B(\nextu, \thisu).
    \]
    Inserting \eqref{eq:abstractbreg:condition}, we obtain \eqref{eq:abstractbreg:fejer}.
    Summing the latter over $k=0,\ldots,N-1$ yields \eqref{eq:abstractbreg:result}.
\end{proof}

\begin{example}
    \label{ex:abstractbreg:prox}
    If $H=\subdiff F$ for a convex function $F$, then by the \hyperref[sec:glossary]{definition} of the convex subdifferential, \eqref{eq:abstractbreg:condition} holds with the gap functional
    \[
        \GenGap(u, \optu) = F(u)-F(\optu).
    \]
    If we take $\optu$ is a minimiser of $F$, then the gap functional is non-negative and indeed positive if $u$ is also not minimiser. This is why it is called a gap functional.

    Consider then for some step length parameter $\tau>0$ the proximal point method \eqref{eq:bregpdps:basicprox} in a Hilbert space $X$, that is, taking $B=\inv\tau N_X$,
    \[
        \nextu \defeq \prox_{\tau F}(\thisx),
        \quad\text{equivalently}\quad
        0 \in \subdiff F(\nextu) + \tau(\nextu-\thisu).
    \]
    Then \eqref{eq:abstractbreg:result} reads
    \begin{equation}
        \label{eq:abstractbreg:result-norm}
        \frac{1}{2\tau}\norm{u^N-\optu}_X^2
        +\sum_{k=0}^{N-1} \frac{1}{2} \norm{\nextu-\thisu}_X^2
        +\sum_{k=0}^{N-1} \tau(F(\nextu)-F(\optu))
        \le
        \frac{1}{2}\norm{\optu-u^0}_X^2.
    \end{equation}
    With $\optu$ a minimiser, this clearly forces $F(u^N) \downto F(\optu)$ as $N \upto \infty$, suggesting why we call \eqref{eq:abstractbreg:result} the “descent inequality”.
    
\end{example}

If our problem is non-convex, then we can try to locally ensure second-order growth by imposing $\GenGap(\nextu, \optu) \ge 0$.
Verifying this for the PDBS will be the topic of \cref{sec:convergence:second,sec:convergence:second:block}.
If $B$ is not given by the standard generating function $N_X$ on a Hilbert spaces $X$, then to get from \eqref{eq:abstractbreg:result} an estimate like \eqref{eq:abstractbreg:result-norm} on norms, we can assume the ellipticity or at least semi-ellipticity of the overall Bregman divergence $B$.
Verifying this for $B=B_{J^0}$ with $J^0$ given in \eqref{eq:bregpdps:j0} is our next topic.

\subsection{Ellipticity of the Bregman divergences}
\label{sec:convergence:ellipticity}

As just discussed, for \cref{thm:abstractbreg} to provide estimates that we can use to prove the convergence of the PDBS, we need at least the semi-ellipticity of $B^0$ generated by $J^0$ given in \eqref{eq:bregpdps:j0}.
Deriving simple conditions that ensure such semi-ellipticity or ellipticity is the topic of the present subsection. To do this, we need the “basic” Bregman divergences $B_X$ and $B_Y$ on both spaces $X$ and $Y$ to be elliptic:

\begin{fixedimportant}{Standing assumption}%
    In this subsection, we assume that $B_X$ is $\inv\tau$-elliptic and $B_Y$ is $\inv\sigma$-elliptic for some $\tau,\sigma>0$. This is true for the Hilbert-space PDPS \eqref{eq:bregpdps:alg:hilbert} where $\tau$ and $\sigma$ are the primal and dual step length parameters.
\end{fixedimportant}

The examples that follow the next general lemma will provide improved estimates.

\begin{lemma}
    \label{lemma:bregpdps:distkplus1-for-k-lowerbound}
    Suppose $K \in C^1(X \times Y)$ is Lipschitz-continuously differentiable with the factor $L_{DK}$ in a convex subdomain $\Omega \subset X \times Y$. Then for $u, u' \in \Omega$,
    \begin{equation}
        \label{eq:bregpdps:bkbound-general}
        B_K(u', u) \le \frac{L_{DK}}{2}\norm{u'-u}_{X \times Y}^2.
    \end{equation}
    Consequently, if $B_X$ is $\inv\tau$-elliptic and $B_Y$ is $\inv\sigma$-elliptic and $1 \ge \max\{\tau,\sigma\}L_{DK}$, then $B^0$ is semi-elliptic (elliptic if the the inequality is strict) within $\Omega$.
\end{lemma}

\begin{proof}
    By definition, $B_K(u', u) = K(u') - K(u) - \dualprod{DK(u)}{u'-u}$.
    Using the mean value equality in $\R$ with the chain rule and the Cauchy--Schwarz inequality, we get
    \[
        B_K(u', u) = \int_0^1 \dualprod{DK(u+t(u'-u))-DK(u)}{u'-u} \d t
        \le \int_0^1 t L_{DK} \norm{u'-u}_{X \times Y}^2 \d t.
    \]
    Calculating the last integral yields \eqref{eq:bregpdps:bkbound-general}.

    For the \mbox{(semi-)}ellipticity, we need $B^0(u, u') \ge \frac{\epsilon}{2}\norm{u-u'}_{X \times Y}^2$ for some $\epsilon > 0$ ($\epsilon=0$) and all $u, u' \in \Omega$.
    Since $B_X$ and $B_Y$ are $\inv\tau$- and $\inv\sigma$-elliptic, we have
    \begin{equation}
        \label{eq:bregpdps:bxy-steplength-ellipticity}
        \begin{aligned}[t]
        B^0(u', u) &= B_X(x', x)+B_Y(y',y)-B_K(u', u)
        \\
        & \ge
        \frac{1}{2\tau}\norm{x'-x}_X^2 + \frac{1}{2\sigma}\norm{y'-y}_Y^2 - B_K(u' u).
        \end{aligned}
    \end{equation}
    Using \eqref{eq:bregpdps:bkbound-general}, therefore
    $
        B^0(u', u) \ge
        \tfrac{\inv\tau-L_{DK}}{2}\norm{x'-x}_X^2
        +
        \tfrac{\inv\sigma-L_{DK}}{2}\norm{y'-y}_Y^2.
    $
    Thus $B^0$ is $\epsilon$-elliptic when $\inv\tau,\inv\sigma \ge L_{DK}+\epsilon$.
    This gives the claim.
\end{proof}

We now provide several examples of ellipticity. In practise, to guarantee ellipticity, we would choose $\tau,\sigma>0$ to satisfy the stated conditions.

\begin{example}
    \label{ex:bregpdps:bk0pos-nonconvex-e}
    Suppose $K(x,y)=E(x)$ with $D E$ $L_{DE}$-Lipschitz in $\Omega=X \times Y$. Then $L_{DK}=L_{DE}$, so we recover the standard-for-gradient-descent step length bound $1 \ge \tau L_{DE}$ for $B^0$ to be semi-elliptic in $\Omega$ (elliptic if the inequality is strict).
\end{example}

\begin{example}
    \label{ex:bregpdps:bk0pos-linear}
    If $K(x,y)=\dualprod{Ax}{y}$ for $A \in \linear(X; Y^*)$, then $B^0$ is elliptic under the standard-for-PDPS \cite{chambolle2010first} step length condition
    \[
        1 > \tau\sigma\norm{A}^2.
    \]
    Indeed, 
    \[
        \dualprod{DK(u+t(u'-u))-DK(u)}{u'-u}
        =2t\dualprod{A(x-x')}{y-y'}.
    \]
    Therefore, taking any $w>1$, we easily improve \eqref{eq:bregpdps:bkbound-general} to
    \begin{equation}
        \label{eq:bregpdps:bk0bound-linear}
        \begin{aligned}[t]
        B_K(u', u) & \le \norm{A}\norm{x'-x}_X \norm{y'-y}_Y
        \\
        &
        \le
        \frac{w \norm{A}}{2}\norm{x'-x}_{X}^2
        +\frac{\inv w \norm{A}}{2}\norm{y'-y}_{Y}^2
        \quad
        (u, u' \in X \times Y).
        \end{aligned}
    \end{equation}
    By \eqref{eq:bregpdps:bxy-steplength-ellipticity}, $B^0$ is therefore $\epsilon$-elliptic
    if $\inv\tau \ge w\norm{A}+\epsilon$ and $\inv\sigma \ge \inv w\norm{A}+\epsilon$.
    Taking $w=\sigma \norm{A}/(1-\sigma\epsilon)$ this holds if $1 \ge \tau\sigma\norm{A}^2/(1-\sigma\epsilon)+\tau\epsilon$. Since $\epsilon>0$ was arbitrary, the claimed step length condition follows.
\end{example}

\begin{example}
    \label{ex:bregpdps:bk0pos-liny}
    Suppose $K(x,y)=\dualprod{A(x)}{y}$ with $A$ and $D A$ Lipschitz with the respective factors $L_A, L_{DA} \ge 0$. Then $B^0$ is elliptic within $\Omega = X \times \B(0, \rho_y)$ if
    \[
        1 > \tau\sigma L_A^2+\tau \frac{L_{DA}\rho_y}{2}.
    \]
    Indeed, for any $w>1$, using the mean value equality as in the proof of \cref{lemma:bregpdps:distkplus1-for-k-lowerbound}, we deduce
    \begin{equation}
        \label{eq:bregpdps:bk0bound-liny}
        \begin{aligned}[t]
        B_K(u', u)
        &
        =
        \dualprod{A(x')-A(x)}{y'}
        -\dualprod{DA(x)(x'-x)}{y}
        \\
        &
        =
        \dualprod{A(x')-A(x)}{y'-y}
        +\dualprod{A(x')-A(x)-DA(x)(x'-x)}{y}
        \\
        &
        \le
        L_A\norm{x'-x}_X\norm{y'-y}_Y
        +\tfrac{L_{DA}\norm{y'}}{2}\norm{x'-x}_X^2\\
        &
        \le
        \frac{w L_A + L_{DA}\norm{y}}{2} \norm{x'-x}_X^2
        +\frac{\inv w L_A}{2}\norm{y'-y}_Y^2.
        \end{aligned}
    \end{equation}
    If $\rho_y>0$ is such that $\norm{y} \le \rho_y$, taking $w=\sigma L_A/(1-\sigma\epsilon)$, similarly to  \cref{ex:bregpdps:bk0pos-linear} we deduce the claimed bound.
\end{example}

We can combine the examples above:

\begin{example}
    \label{ex:bregpdps:bk0pos-combo}
    As in \cref{ex:bregpdps:twoblock}, take $K(x,(y_1,y_2))=\dualprod{A_1(x)}{y_1}+\dualprod{A_2 x}{y_2}$ with $A_1 \in C^1(X; Y_1^*)$ and $A_2 \in \linear(X; Y_2^*)$.
    Then $B^0$ is elliptic within $\Omega = X \times \B(0, \rho_y)$ if
    \[
        1 > \tau\sigma(L_{A_1}^2+\norm{A_2}^2)+\tau \frac{L_{DA_1}\rho_{y_1}}{2}.
    \]
    Indeed, we bound $B_K$ by summing \eqref{eq:bregpdps:bk0bound-linear} for $A_1$ and \eqref{eq:bregpdps:bk0bound-liny} for $A_2$. This yields for any $w_1, w_2>0$ the estimate
    \begin{equation}
        \label{eq:bregpdps:bk0pos-combo-bk}
        \begin{aligned}[t]
            B_K(u', u)
            &
            \le
            \frac{w_1 L_{A_1} + L_{DA_1}\norm{y_1}}{2} \norm{x-x'}_X^2
            +\frac{\inv w_1 L_{A_1}}{2}\norm{y_1'-y_1}_Y^2
            \\\MoveEqLeft[-1]
            +\frac{w_2 \norm{A_2}}{2}\norm{x'-x}_{X}^2
            +\frac{\inv w_2 \norm{A_2}}{2}\norm{y_2'-y_2}_{Y_2}^2.
        \end{aligned}
    \end{equation}
    Taking $w_1=\sigma L_{A_1}/(1-\sigma\epsilon)$ and $w_2=\sigma \norm{A_2}/(1-\sigma\epsilon)$, and using \eqref{eq:bregpdps:bxy-steplength-ellipticity}, we deduce the claimed ellipticity for small enough $\epsilon>0$.
\end{example}

\begin{remark}
    \label{rem:convergence:ellipticity:bounds}
    In  \cref{ex:bregpdps:bk0pos-liny,ex:bregpdps:bk0pos-combo} we needed a bound on the dual variable $y$. In the latter, as an improvement, this was only needed on the subspace $Y_1$ of non-bilinearity.
    An ad-hoc solution is to introduce the bound into the problem. In the Hilbert case, \cite{tuomov-nlpdhgm-redo,tuomov-nlpdhgm-general} secure such bounds by taking the primal step length $\tau$ small enough and arguing as in \cref{thm:abstractbreg} individually on the primal and dual iterates.
\end{remark}

\subsection{Ellipticity for block-adapted methods}
\label{sec:convergence:ellipticity:block}

\index{primal-dual!splitting!block-adapted|(}%
\index{PDPS!block-adapted|(}%

We now study ellipticity for block-adapted methods. The goal is to obtain faster convergence by adapting the blockwise step length parameters to the problem structure (connections between blocks) and the local (blockwise) properties of the problem.

\begin{fixedimportant}{Standing assumption}%
    In this subsection, we assume $F$, $G_*$, $J_X$ and $J_Y$ to have the form of \cref{sec:bregpdps:block}. In particular, $X$ and $Y$ are (products of) Hilbert spaces, and
    \begin{equation}
        \label{eq:bregpdps:block0-bxy}
        B^0(u', u) = \sum_{j=1}^m \frac{1}{2\tau_j} \norm{x_j'-x_j}_{X_j}^2
        +\sum_{\ell=1}^n \frac{1}{2\sigma_\ell} \norm{y_\ell'-y_\ell}_{Y_\ell}^2 - B_K(u', u).
    \end{equation}
\end{fixedimportant}

We start by refining the two-block \cref{ex:bregpdps:bk0pos-combo} to be adapted to the blocks:

\begin{example}
    \label{ex:bregpdps:twoblock:steps}
    Let  $K(x,(y_1,y_2))=\dualprod{A_1(x)}{y_1}+\dualprod{A_2 x}{y_2}$ with $A_1 \in C^1(X; Y_1^*)$ and $A_2 \in \linear(X; Y_2^*)$ as in \cref{ex:bregpdps:bk0pos-combo,ex:bregpdps:twoblock}. Write $\tau=\tau_1$.
    Using \eqref{eq:bregpdps:bk0pos-combo-bk} in \eqref{eq:bregpdps:block0-bxy} for $m=1$ and $n=2$ with \eqref{eq:bregpdps:bk0pos-combo-bk}, we see $B^0$ to be $\epsilon$-elliptic within $\Omega=X \times \B(0, \rho_{y_1}) \times Y_2$ if $\inv\tau \ge w_1 L_{A_1} + L_{DA_1}\rho_{y_1} + w_2\norm{A_2} + \epsilon$ and $\inv\sigma_1 \ge \inv w_1 L_{A_1}$ as well as $\inv\sigma_2 \ge \inv w_2\norm{A_2} + \epsilon$.  Taking $w_1=\sigma_1 L_{A_1}/(1-\sigma_1\epsilon)$ and $w_2=\sigma_2 \norm{A_2}/(1-\sigma_2\epsilon)$, $B^0$ is therefore elliptic (some $\epsilon>0$) within $\Omega$ if
    $
        1 > \tau(\sigma_1 L_{A_1}^2+\sigma_2 \norm{A_2}^2)+\tau \frac{L_{DA_1}\rho_{y_1}}{2}.
    $
\end{example}

\begin{example}
    In \cref{ex:bregpdps:twoblock:steps}, if both $A_1 \in \linear(X; Y_1^*)$ and $A_2 \in \linear(X; Y_2^*)$, then $B^0$ is elliptic within $\Omega=X\times Y_1 \times Y_2$ if $1 > \tau(\sigma_1 \norm{A_1}^2+\sigma_2 \norm{A_2}^2)$.
\end{example}

\begin{example}
    \label{ex:bregpdps:block0}
    Suppose we can write $K(x,y)=\sum_{j=1}^m \sum_{\ell=1}^n K_{j\ell}(x_j, y_\ell)$ with each $K_{j\ell}$ Lipschitz-continuously differentiable with the factor $L_{j\ell}$.
    Following  \cref{lemma:bregpdps:distkplus1-for-k-lowerbound},
    \begin{equation}
        \label{eq:bregpdps:block0-k}
        B_K(u', u) \le \sum_{j=1}^m \sum_{\ell=1}^n \frac{L_{j\ell}}{2}(\norm{x_j'-x_j}^2+\norm{y_\ell'+y_\ell}^2).
    \end{equation}
    Consequently, using \eqref{eq:bregpdps:block0-bxy}, we see that $B^0$ is $\epsilon$-elliptic if $1 \ge \tau_j(\sum_{\ell=1}^n L_{j\ell}+\epsilon)$ and $1 \ge \sigma_\ell(\sum_{j=1}^n L_{j\ell}+\epsilon)$ for all $j=1,\ldots,m$ and $\ell=1,\ldots,n$.
\end{example}

\begin{example}
    \label{ex:bregpdps:block0linear}
    If $K(x,y)=\sum_{j=1}^m \sum_{\ell=1}^m \dualprod{A_{j\ell}x_j}{y_\ell}$ for some $A_{j\ell} \in \linear(X_j; Y_\ell^*)$, then following \cref{ex:bregpdps:bk0pos-linear}, for arbitrary $w_{j\ell}>0$,
    \[
        \begin{aligned}[t]
        B_K(u', u) & \le \sum_{j=1}^m\sum_{\ell=1}^m \norm{A_{j\ell}}\norm{x_j'-x_j}\norm{y_j'-y_j}
        \\ &
        \le \sum_{j=1}^m\sum_{\ell=1}^n
        \left( \frac{w_{j\ell}\norm{A_{j\ell}}}{2}\norm{x_j'-x_j}^2
                +\frac{\inv w_{j\ell}\norm{A_{j\ell}}}{2}\norm{y_\ell'-x_\ell}^2\right).
        \end{aligned}
    \]
    Using \eqref{eq:bregpdps:block0-bxy}, $B^0$ is thus $\epsilon$-elliptic if $1 \ge \tau_j(\epsilon+\sum_{\ell=1}^n w_{j\ell}\norm{A_{j\ell}})$ and  $1 \ge \sigma_\ell(\epsilon+\sum_{j=1}^m \inv w_{j\ell}\norm{A_{j\ell}})$ for all $j=1,\ldots,m$ and $\ell=1,\ldots,n$. We can use the factors $w_{j\ell}$ to adapt the algorithm to the different blocks for potentially better convergence.
\end{example}

\index{primal-dual!splitting!block-adapted|)}%
\index{PDPS!block-adapted|)}%

\subsection{Non-smooth second-order conditions}
\label{sec:convergence:second}

We now study conditions for \eqref{eq:abstractbreg:condition} to hold with $\GenGap(\freevar,\optu) \ge 0$.
We start by writing out the condition for the PDBS.

\begin{lemma}
    \label{lemma:bregpdps:convergence}
    Let $\optu=(\optx,\opty) \in X \times Y$ and suppose for some $\GenGap(u, \optu) \in \R$ and a neighbourhood $\Omega_\optu \subset X \times Y$ that for all $u=(x,y) \in \Omega_\optu$, $x^* \in \subdiff F(x)$, and $y^* \in \subdiff G_*(y)$,
    \begin{equation}
        \label{eq:bregpdps:condition2}
        \tag{C$^2$}
        \dualprod{x^*+D_x K(x, y)}{x-\optx}
        +\dualprod{y^*-D_y K(x, y)}{y-\opty}
        \ge
        \GenGap(u, \optu).
    \end{equation}
    Let $\{\nextu\}_{k \in \N}$ be generated by the PDBS \eqref{eq:bregpdps:alg:expanded} for some $u^0 \in X \times Y$, and suppose $\{\thisu\}_{k \in \N} \subset \Omega_\optu$.
    Then with $B=B^0$ the fundamental condition \eqref{eq:abstractbreg:condition} and the quantitative $\Delta$-Féjer monotonicity \eqref{eq:abstractbreg:fejer} hold for all $k \in \N$, and the descent inequality \eqref{eq:abstractbreg:result} holds for all $N \ge 1$.
\end{lemma}

\begin{proof}
    \Cref{thm:abstractbreg} proves \eqref{eq:abstractbreg:fejer} and \eqref{eq:abstractbreg:result} if we show \eqref{eq:bregpdps:condition2}.
    For $H$ in \eqref{eq:convergence:h}, we have
    \begin{equation*}
        \nexxt h =
        \begin{pmatrix}
            x_{k+1}^* +D_x K(\nextx, \nexty) \\
            y_{k+1}^* - D_y K(\nextx, \nexty)
        \end{pmatrix}
        \in H(\nextu)
        \quad\text{with}\quad%
        \left\{
        \begin{array}{rl}
            x_{k+1}^* & \in \subdiff F(\nextx),
            \\
            y_{k+1}^* & \in \subdiff G_*(\nexty).
        \end{array}
        \right.
    \end{equation*}
    Thus \eqref{eq:abstractbreg:condition} expands as \eqref{eq:bregpdps:condition2} for $u=\nextu$ and $(x^*,y^*)=(x_{k+1}^*,y_{k+1}^*)$.
\end{proof}

In \cref{sec:convergence:gaps} on the convergence of gap functionals, we will consider general $\optu$ in \cref{eq:bregpdps:condition2}. For the moment, we however fix a root $\optu=\realoptu \in \inv H(0)$. Then
\begin{equation}
    \label{eq:bregpdps:ocstar}
    0 =
    \begin{pmatrix}
        \realoptx^* + D_x K(\realoptx, \realopty) \\
        \realopty^* - D_y K(\realoptx, \realopty)
    \end{pmatrix} \in H(\realoptu)
    \quad\text{with}\quad%
    \left\{
    \begin{array}{rl}
        \realoptx^* & \in \subdiff F(\realoptx),
        \\
        \realopty^* & \in \subdiff G_*(\realopty).
    \end{array}
    \right.
\end{equation}
Since we assume $F$ and $G_*$ to be convex, their subdifferentials are monotone. When $K$ is not convex-concave, and to obtain strong convergence of iterates even when it is, we will  need some strong monotonicity of the subdifferentials, but only \emph{at} a solution.
Specifically, for $\gamma > 0$, we say that $T: X \setto X^*$ is \term[monotone!strongly]{$\gamma$-strongly monotone} \emph{at} $\realoptx$ for $\realoptx^* \in T(\realoptx)$ if
\begin{equation}
    \label{eq:bregpdps:strongly-monotone}
    \dualprod{x^*-\realoptx^*}{x-\realoptx} \ge \gamma\norm{x-\realopt x}_X^2
    \quad (x \in X,\, x^* \in T(x)).
\end{equation}
If $\gamma=0$, we drop the word “strong”.
For $T=\subdiff F$, \eqref{eq:bregpdps:strongly-monotone} follows from the $\gamma$-strong subdifferentiability of $F$.

\begin{fixedimportant}{Standing assumption}%
    Throughout the rest of this subsection, we assume \eqref{eq:bregpdps:ocstar} to hold and that $\subdiff F$ is ($\gamma_F$-strongly) monotone at $\realoptx$ for $\realoptx^*$, and $\subdiff G_*$ is ($\gamma_{G_*}$-strongly) monotone at $\realopty$ for $\realopty^*$.
\end{fixedimportant}

\begin{lemma}
    \label{lemma:bregpdps:second-order-cond}
    The nonsmooth second-order growth condition \eqref{eq:bregpdps:condition2} holds provided
    \begin{gather}
        \label{eq:bregpdps:second-order-cond-symbreg}
        \gamma_F\norm{x-\realoptx}^2
        +\gamma_{G_*}\norm{y-\realopty}^2
        \ge
        B_K(\realoptu,u)+B_K(u, \realoptu)
        + \GenGap(u, \realoptu)
        \quad (u \in \Omega_\optu),
        \\
        \shortintertext{equivalently}
        \label{eq:bregpdps:second-order-cond-dk}
        \tag{\ref*{eq:bregpdps:second-order-cond-symbreg}$^\prime$}
        \gamma_F\norm{x-\realoptx}^2
        +\gamma_{G_*}\norm{y-\realopty}^2
        \ge
        a_K(\realoptu,u)+a_K(u, \realoptu)
        + \GenGap(u, \realoptu)
        \quad (u \in \Omega_\optu)
        \\
        \shortintertext{for}
        \label{eq:bregpdps:dk-def}
        a_K(u, \optu) \defeq K(x,y) - K(\optx, \opty) + \dualprod{D_x K(x, y)}{\optx-x} + \dualprod{D_y K(\optx,\opty)}{\opty-y}.
    \end{gather}
\end{lemma}

Note that \eqref{eq:bregpdps:second-order-cond-symbreg} involves the \term[Bregman divergence!symmetrised]{symmetrised Bregman divergence} $B_K^S(u, u') \defeq B_K(u, u') + B_K(u', u)$ generated by $K$.

\begin{proof}
    Inserting the zero of \eqref{eq:bregpdps:ocstar} in \eqref{eq:bregpdps:condition2}, we rewrite the latter as
    \begin{multline*}
        \dualprod{x^*- \realoptx^*}{x-\realoptx}
        +\dualprod{y^*-\realopty^*}{y-\realopty}
        \ge
        \dualprod{D_x K(x,y)-D_x K(\realoptx, \realopty)}{\realoptx-x}
        \\
        +\dualprod{D_y K(x,y)-D_y K(\realoptx, \realopty)}{y-\realopty}
        + \GenGap(\nextu, \realoptu).
    \end{multline*}
    Using the assumed strong monotonicities, and the definitions of $B_K$ and $a_K$, this is immediately seen to hold when \eqref{eq:bregpdps:second-order-cond-symbreg} or \eqref{eq:bregpdps:second-order-cond-dk} does.
\end{proof}


\begin{example}
    \label{ex:bregpdps:second-order-cond:bilinear}
    If $K$ is convex-concave, the next \cref{lemma:bregpdps:dk} and \cref{lemma:bregpdps:second-order-cond} prove \eqref{eq:bregpdps:condition2} for
    \[
        \GenGap(u, \realoptu)=\gamma_F\norm{x-\realoptx}^2+\gamma_{G_*}\norm{y-\realopty}^2 \ge 0
        \quad\text{within}\quad
        \Omega_\realoptu=X \times Y.
    \]
    This is in particular true for $K(x,y)=\dualprod{Ax}{y}+E(x)$ with $A \in \linear(X; Y^*)$ and $E \in C^1(X)$ convex.
\end{example}

\begin{lemma}
    \label{lemma:bregpdps:dk}
    Suppose $K: X \times Y \to \R$ is Gâteaux-differentiable and convex-concave. Then $a_K(u, \optu) \le 0$ and $B_K^S(u, \optu) \le 0$ for all $u, \optu \in X \times Y$.
\end{lemma}

\begin{proof}
    The convexity of $K(\freevar, y)$ and the concavity of $K(\optx, \freevar)$ show
    \begin{align*}
        K(x, y)-K(\optx, y) + \dualprod{D_xK(x,y)}{\optx-x}
        &\le 0 \quad\text{and} \\
        K(\optx, y)-K(\optx, \opt y) + \dualprod{D_yK(\optx,\opty)}{\opty-y} & \le 0.
    \end{align*}
    Summing these two estimates proves $a_K(u, \optu) \le 0$, consequently $B_K^S(u, \optu)=a_K(u, \optu)+a_K(\optu, u) \le 0$.
\end{proof}

\begin{example}
    \label{ex:bregpdps:second-order-cond:general-lipschitz}
    Suppose $K$ has $L_{DK}$-Lipschitz derivative within $\Omega \subset X \times Y$. If $\realoptu \in \Omega$, then by  \cref{lemma:bregpdps:distkplus1-for-k-lowerbound}, $B_K(u,\realoptu),B_K(\realoptu, u) \le \tfrac{L_{DK}}{2}\norm{u-\realoptu}_{X \times Y}^2$ for $u \in \Omega$. Thus \eqref{eq:bregpdps:condition2} holds by \cref{lemma:bregpdps:second-order-cond} with $\Omega_\realoptu=\Omega$ and
    \[
        \GenGap(u, \realoptu)=(\gamma_F-L_{DK})\norm{x-\realoptx}^2+(\gamma_{G_*}-L_{DK})\norm{y-\realopty}^2.
    \]
    This is non-negative if $\gamma_F,\gamma_{G_*} \ge L_{DK}$.
\end{example}

\begin{example}
    \label{ex:bregpdps:second-order-cond:liny}
    Let $K(x,y)=\dualprod{A(x)}{y}$ for some $A \in \linear(X; Y^*)$ such that $DA$ is Lipschitz with the factor $L_{DA} \ge 0$.
    For some $\tilde\gamma_F,\tilde\gamma_{G_*} \ge 0$ and $\rho_y, \realopt\rho_x, \alpha>0$, let either
    \begin{enumerate}[label=(\alph*)]
        \item\label{item:bregpdps:second-order-cond:primalonly} $\tilde\gamma_F \ge \frac{L_{DA}}{2}(\rho_y+\norm{\realopty}_Y)$, $\tilde\gamma_{G_*} \ge 0$, and $\Omega_\realoptu=X \times \B(0, \rho_y)$; or
        \item\label{item:bregpdps:second-order-cond:primaldual} $\tilde\gamma_F > L_{DA}\left(\norm{\realopty}_Y + \frac{\alpha}{2}\right)$, $\tilde\gamma_{G_*} \ge \frac{L_{DA}}{2\alpha}\realopt \rho_x^2$, and $\Omega_\realoptu= \B(\realoptx, \realopt\rho_x) \times Y$.
    \end{enumerate}
    Then  \cref{lemma:bregpdps:second-order-cond} proves \eqref{eq:bregpdps:condition2} with
    \[
        \GenGap(u, \realoptu)=(\gamma_F-\tilde\gamma_F)\norm{x-\realoptx}^2+(\gamma_{G_*}-\tilde\gamma_{G_*})\norm{y-\realopty}^2.
    \]

    To see this, we need to prove \eqref{eq:bregpdps:second-order-cond-dk}. Now
    \begin{equation}
        \label{eq:ex:bregpdps:second-order-cond:liny:dk}
        a_K(u, \realoptu) \defeq \dualprod{A(x)-A(\realoptx)+D A(x)(\realoptx-x)}{y}
        \quad (u, \realoptu \in X \times Y).
    \end{equation}
    Arguing with the mean value equality and the Lipschitz assumption as in  \cref{lemma:bregpdps:distkplus1-for-k-lowerbound}, we get $a_K(\realoptu,u)+a_K(u, \realoptu) \le \frac{L_{DA}}{2}(\norm{y}_Y+\norm{\realopty}_Y)\norm{x-\realoptx}^2$. Thus \ref{item:bregpdps:second-order-cond:primalonly} implies \eqref{eq:bregpdps:second-order-cond-dk}.
    By \eqref{eq:ex:bregpdps:second-order-cond:liny:dk}, the mean-value equality, and the Lipschitz assumption, also
    \[
        \begin{aligned}[t]
        a_K(u,\realoptu)+a_K(\realoptu,u)
        &
        =\dualprod{[D A(x)-D A(\realoptx)](\realoptx-x)}{\realopty}
        \\
        \MoveEqLeft[-1]
        +\dualprod{A(x)-A(\realoptx)+D A(x)(\realoptx-x)}{y-\realopty}
        \\
        &
        \le
        L_{DA}\norm{x-\realoptx}_X^2\bigl(
            \norm{\realopty}_Y + \tfrac{1}{2}\norm{y-\realopty}_Y
        \bigr).
        \end{aligned}
    \]
    Using Cauchy's inequality and \ref{item:bregpdps:second-order-cond:primaldual} we deduce \eqref{eq:bregpdps:second-order-cond-dk}.
\end{example}

\begin{remark}
    \label{rem:convergence:second:bounds}
    In the last two examples, we need to bound some of the iterates, and to initialise close enough to a solution. Showing that the iterates stay in a local neighbourhood is a large part of the work in \cite{tuomov-nlpdhgm-redo,tuomov-nlpdhgm-general}, as discussed in \cref{rem:convergence:ellipticity:bounds}.
\end{remark}

\subsection{Second-order growth conditions for block-adapted methods}
\label{sec:convergence:second:block}

\index{primal-dual!splitting!block-adapted|(}%
\index{PDPS!block-adapted|(}%

We now study second-order growth for problems with a block structure as in \cref{sec:bregpdps:block}:

\begin{fixedimportant}{Standing assumption}%
    In this subsection, $F$ and $G_*$ are as in \cref{sec:bregpdps:block}, each component subdifferential $\subdiff F_j$ now ($\gamma_{F_j}$-strongly) monotone at $\realoptx_j$ for $\realoptx_j^*$ and each $\subdiff G_{\ell*}$ ($\gamma_{G_{\ell*}}$-strongly) monotone at $\realopty_\ell$ for $\realopty_\ell^*$.
    Here $\realoptx_j$, $\realoptx_j^*$, $\realopty_\ell$ and $\realopty_\ell^*$ are the components of $\realoptx$, $\realoptx^*$, $\realopty$, and $\realopty^*$ in the corresponding subspace, assumed to satisfy the critical point condition \eqref{eq:bregpdps:ocstar}.
\end{fixedimportant}

As only some of the component functions may have $\gamma_{F_j}, \gamma_{G_{\ell*}} > 0$, through detailed analysis of the block structure, we hope to obtain (strong) convergence on some subspaces even if the entire primal or dual variables might not converge.

Similarly to \cref{lemma:bregpdps:second-order-cond} we prove:

\begin{lemma}
    \label{lemma:bregpdps:second-order-cond:block}
    Suppose for some neighbourhood $\Omega_\realoptu \subset X \times Y$ that
    \[
        \Delta_{k+1} \defeq
        \sum_{j=1}^m \tilde\gamma_{F_j} \norm{x_j-\realoptx_j}_{X_j}^2 +  \sum_{\ell=1}^n \tilde\gamma_{G_{\ell*}} \norm{y_\ell-\realopty_\ell}_{Y_\ell}^2
        \ge
        a_K(\realoptu,u)+a_K(u, \realoptu)
    \]
    for some $\tilde\gamma_{F_j}, \gamma_{G_{\ell*}} \ge 0$ for all $u \in \Omega_\realoptu$.
    Then \eqref{eq:bregpdps:condition2} holds with
    \begin{equation}
        \label{eq:convergence:second:block:gap}
        \GenGap(u, \realoptu) = \sum_{j=1}^m (\gamma_{F_j}-\tilde\gamma_{F_j}) \norm{x_j-\realoptx_j}_{X_j}^2 +  \sum_{\ell=1}^n (\gamma_{G_{\ell*}}-\tilde\gamma_{G_{\ell*}}) \norm{y_\ell-\realopty_\ell}_{Y_\ell}^2.
    \end{equation}
\end{lemma}

In the convex--concave case, we can transfer all strong monotonicity into $\GenGap$:

\begin{example}
    \label{ex:convergence:second:block:convex-concave}
    If $K$ is convex-concave, then by \cref{lemma:bregpdps:dk,lemma:bregpdps:second-order-cond:block}, \eqref{eq:bregpdps:condition2} holds with $\Omega_\realoptu=X \times Y$  and $\GenGap$ as in \eqref{eq:convergence:second:block:gap} for $\tilde\gamma_{F_j}=0$ and $\tilde\gamma_{G_{\ell*}}=0$.
    We have $\GenGap(\freevar, \realoptu) \ge 0$ always.
\end{example}

\begin{example}
    As in \cref{ex:bregpdps:block0}, suppose we can write $K(x,y)=\sum_{j=1}^m \sum_{\ell=1}^n K_{j\ell}(x_j, y_\ell)$ with each $K_{j\ell}$ Lipschitz-continuously differentiable with the factor $L_{j\ell}$ in $\Omega$.
    Then using \eqref{eq:bregpdps:block0-k} and \cref{lemma:bregpdps:second-order-cond:block}, we see \eqref{eq:bregpdps:condition2} to hold with $\Omega_\realoptu=\Omega$ and $\GenGap$ as in \eqref{eq:convergence:second:block:gap} with
    \[
        \tilde\gamma_{F_j} = \sum_{\ell=1}^n L_{j\ell}
        \quad (j=1,\ldots,m)
        \quad\text{and}\quad
        \tilde\gamma_{G_{\ell*}} = \sum_{j=1}^m L_{j\ell}
        \quad (\ell=1,\ldots,n).
    \]
    Thus $\GenGap(\freevar, \realoptu) \ge 0$ if $\gamma_{F_j} \ge  \sum_{\ell=1}^n L_{j\ell}$ and $\gamma_{G_{\ell*}} \ge \sum_{j=1}^m L_{j\ell}$ for all $\ell$ and $j$.
\end{example}

The special case of \cref{ex:bregpdps:block0} with each $K_{j\ell}$ bilinear, corresponding to \cref{ex:bregpdps:block0linear} for ellipticity, is covered by \cref{ex:convergence:second:block:convex-concave}.

We consider in detail the two dual block setup of \cref{ex:bregpdps:twoblock,ex:bregpdps:twoblock:steps}:

\begin{example}
    \label{ex:bregpdps:twoblock:nbd}
    As in \cref{ex:bregpdps:twoblock}, let $K(x,y)=\dualprod{A_1(x)}{y_1}+\dualprod{A_2x}{y_2}$ for $A_1 \in C^1(X; Y_1^*)$ and $A_2 \in \linear(X; Y_2^*)$.
    Then, as in \eqref{eq:ex:bregpdps:second-order-cond:liny:dk},
    \[
        a_K(u, \bar u)=\dualprod{A_1(x)-A_1(\bar x)+DA_1(x)(\bar x-x)}{y_1},
    \]
    which does not depend on $A_2$. For any $\alpha, \rho_y,\realopt\rho_x>0$ let either
    \begin{enumerate}[label=(\alph*)]
        \item $\tilde\gamma_F \ge \frac{L_{DA_1}}{2}(\rho_{y_1}+\norm{\realopty_1}_{Y_1})$, $\tilde\gamma_{G_{1*}} \ge 0$, and $\Omega_\realoptu=X \times \B(0, \rho_{y_1})$; or
        \item $\tilde\gamma_F > L_{DA_1}\left(\norm{\realopty_1}_{Y_1} + \frac{\alpha}{2}\right)$, $\tilde\gamma_{G_{1*}} \ge \frac{L_{DA_1}}{2\alpha}\realopt\rho_x^2$, and $\Omega_\realoptu=\B(\realoptx, \realopt\rho_x) \times Y$.
    \end{enumerate}
    Arguing as in \cref{ex:bregpdps:second-order-cond:liny} and using \cref{lemma:bregpdps:second-order-cond:block}, we then see \eqref{eq:bregpdps:condition2} to hold with  $\GenGap$ as in \eqref{eq:convergence:second:block:gap} and $\tilde\gamma_{G_{2*}}=0$.
    In this case $\GenGap(\freevar, \realoptu)$ is non-negative if $\gamma_F \ge \tilde\gamma_F$ and $\gamma_{G_{1*}} \ge \tilde\gamma_{G_{1*}}$.
\end{example}

\index{primal-dual!splitting!block-adapted|)}%
\index{PDPS!block-adapted|)}%

\subsection{Convergence of iterates}
\label{sec:convergence:iterate}

We are now ready to prove the convergence of the iterates. We start with weak convergence and proceed to strong and linear convergence. For weak convergence in infinite dimensions, we need some further technical assumptions.
We recall that a set-valued map $T: X \setto X^*$ is weak-to-strong (weak-$*$-to-strong) outer semicontinuous if $x_k^* \in T(x^k)$ and $x^k \weakto x$ ($x^k \weaktostar x$) and $x_k^* \to x^*$ imply $x^* \in T(x)$. The non-reflexive case of the next assumption covers spaces of functions of bounded variation \cite[Remark 3.12]{ambrosio2000fbv}, important for total variation based imaging.

\begin{assumption}
    \label{ass:convergence:weak}
    Each of the spaces $X$ and $Y$ is, individually, either a reflexive Banach space or the dual of separable space.
    The operator $H: X \times Y \setto X^* \times Y^*$ is \mbox{weak(-$*$)}-to-strong \indexalso{outer semicontinuous}, where we mean by ``weak(-$*$)'' that we take the weak topology if the space is reflexive and weak-$*$ otherwise, individually on $X$ and $Y$.
\end{assumption}

Subdifferentials of lower semicontinuous convex functions are weak(-$*$)-to-strong outer semicontinuous\footnote{This result seems difficult to find in the literature for Banach spaces, but follows easily from the definition of the subdifferential: If $F(x) \ge F(\thisx) + \dualprod{x_k^*}{x-\thisx}$ and $x_k^* \to \realoptx^*$ as well as $\thisx \weakto $ (or $\weaktostar$) $\realoptx$, then, using the fact that $\{\norm{\thisx-\realoptx}\}_{k \in \N}$ is bounded, in the limit $F(x) \ge F(\realoptx) + \dualprod{\realoptx^*}{x-\realoptx}$.}, so the outer semicontinuity of $H$ depends mainly on $K$.

\begin{example}
    If $X$ and $Y$ are finite-dimensional, \cref{ass:convergence:weak} holds if $K \in C^1(X; Y)$.
\end{example}

\begin{example}
    More generally, \cref{ass:convergence:weak} holds if $K \in C^1(X \times Y)$ and $DK$ is continuous from the weak(-$*$) topology to the strong topology.
\end{example}

\begin{example}
    If $K=\dualprod{Ax}{y} + E(x)$ for $A \in \linear(X; Y^*)$ and $E \in C^1(X)$ convex, then  $H$ satisfies \cref{ass:convergence:weak}. Indeed, it can be shown that $H$ is maximal monotone, hence weak(-$*$) outer semicontinuous similarly to convex subdifferentials.
\end{example}

\begin{fixedimportant}{Verification of the conditions}%
    To verify the nonsmooth second-order growth condition \eqref{eq:bregpdps:condition2} for each of the following \cref{thm:convergence:weak,thm:convergence:strong,thm:convergence:linear}, we point to \cref{sec:convergence:second,sec:convergence:second:block}.
    For the verification of the \mbox{(semi-)}ellipticity of $B^0$, we point to \cref{sec:convergence:ellipticity,sec:convergence:ellipticity:block}.
    As special cases of the PDBS \eqref{eq:bregpdps:alg:expanded}, the theorems apply to the Hilbert-space PDPS \eqref{eq:bregpdps:alg:hilbert} and its block-adaptation \eqref{eq:bregpdps:alg:block}. Then $J_X$ and $J_Y$ are continuously differentiable and convex.
\end{fixedimportant}

\begin{theorem}[Weak convergence]\index{convergence!weak}
    \label{thm:convergence:weak}
    Let $F$ and $G_*$ be convex, proper, and lower semicontinuous; $K \in C^1(X \times Y)$; and both $J_X \in C^1(X)$ and $J_Y \in C^1(Y)$ convex. 
    Suppose \cref{ass:convergence:weak} holds and for some $\realoptu \in \inv H(0)$ that
    \begin{enumerate}[label=(\roman*)]
        \item \eqref{eq:bregpdps:condition2} holds with $\GenGap(\freevar,\realoptu) \ge 0$ within $\Omega_\realoptu \subset X \times Y$; and
        \item $B^0$ is elliptic within $\Omega \ni \realoptu$.
    \end{enumerate}
    Let $\{\nextu\}_{k \in \N}$ be generated by the PDBS \eqref{eq:bregpdps:alg:expanded} for any initial $u^0$, and suppose $\{\thisu\}_{k \in \N} \subset \Omega \isect \Omega_\realoptu$.
    Then there exists at least one cluster point of $\{\thisu\}_{k \in \N}$, and all weak(-$*$) cluster points belong to $\inv H(0)$.
\end{theorem}

\begin{proof}
    \Cref{lemma:bregpdps:convergence} establishes \eqref{eq:abstractbreg:result} for $B=B^0$ and all $N \ge 1$. With
    $\epsilon>0$ the factor of ellipticity of $B^0$, it follows
    \[
        \frac{\epsilon}{2}\norm{u^N-\realoptu}_{X \times Y}^2
        +
        \frac{\epsilon}{2}\sum_{k=0}^{N-1} \norm{\nextu-\thisu}_{X \times Y}^2
        \le B^0(\realoptu, u^0)
        \quad (N \ge 1).
    \]
    Clearly $\norm{\nextu-\thisu} \to 0$ while $\{\norm{u^{k}-\realoptu}\}_{k \in \N}$ is bounded. Using the Eberlein--{\u S}mulyan theorem in a reflexive $X$ or $Y$, and the Banach--Alaoglu theorem otherwise ($X$ or $Y$ the dual of a separable space), we may therefore find a subsequence of $\{u^{k}\}_{k \in \N}$ converging weakly(-$*$) to some $\optx$.
    Since $J^0 \in C^1(X \times Y)$, we deduce $D_1 B^0(\nextu, \thisu) \to 0$.
    Consequently \eqref{eq:bregpdps:alg} implies that $0 \in \limsup_{k \to \infty} H(\nextu)$, where we write ``$\limsup$'' for the Painlevé--Kuratowski outer limit of a sequence of sets in the strong topology.
    Since $H$ is weak(-$*$)-to-strong  outer semicontinuous by \cref{ass:convergence:weak}, it follows that $0 \in H(\realoptu)$.
    Therefore, there exists at least one cluster point of $\{\thisu\}_{k \in \N}$ belonging to $\inv H(0)$.
    Repeating the argument on any  weak(-$*$) convergent subsequence, we deduce that all cluster points belong to $\inv H(0)$.
\end{proof}

\begin{remark}
    For a unique weak limit we may in Hilbert spaces use the quantitative Féjer monotonicity \eqref{eq:abstractbreg:fejer} with \index{lemma!Opial}Opial's lemma \cite{opial1967weak,browder1967convergence}. For bilinear $K$ the result is relatively immediate, as $B^0$ is a squared matrix-weighted norm; see \cite{tuomov-proxtest}. Otherwise a variable-metric Opial's lemma  \cite{tuomov-nlpdhgm-redo} and additional work based on the \index{lemma!Brezis--Crandall--Pazy}Brezis--Crandall--Pazy lemma \cite[Corollary 20.59 (iii)]{bresiz1970perturbations} is required; see \cite{tuomov-nlpdhgm-redo} for $K(x,y)=\dualprod{A(x)}{y}$, and \cite{tuomov-nlpdhgm-general} for general $K$.
\end{remark}

\begin{theorem}[Strong convergence]\index{convergence!strong}
    \label{thm:convergence:strong}
    Let $F$ and $G_*$ be convex, proper, and lower semicontinuous; $K \in C^1(X \times Y)$; and both $J_X \in C(X)$ and $J_Y \in C(Y)$ convex and Gâteaux-differentiable. Suppose for some $\realoptu \in \inv H(0)$ that
    \begin{enumerate}[label=(\roman*)]
        \item \eqref{eq:bregpdps:condition2} holds with $\GenGap(\freevar,\realoptu) \ge 0$ within $\Omega_\realoptu \subset X \times Y$; and
        \item $B^0$ is semi-elliptic within $\Omega \ni \realoptu$.
    \end{enumerate}
    Let $\{\nextu\}_{k \in \N}$ be generated by the PDBS \eqref{eq:bregpdps:alg:expanded} for any initial $u^0$, and suppose $\{\thisu\}_{k \in \N} \subset \Omega \isect \Omega_\realoptu$.
    Then $\GenGap(\nextu, \realoptu) \to 0$ as $N \to \infty$.
    
    In particular, if $\GenGap(u, \realoptu) \ge \norm{P(u-\realoptu)}_Z^2$ for some $P \in \linear(X; Z)$, then $P x^N \to P \realoptx$ strongly in $Z$ and the \term{ergodic sequence} $\tilde x^N_P \defeq \frac{1}{N}\sum_{k=0}^{N-1} P \nextx \to P \realoptx$ at rate $O(1/N)$.
\end{theorem}

\begin{proof}
    \Cref{lemma:bregpdps:convergence} establishes \eqref{eq:abstractbreg:result}. By the semi-ellipticity of $B^0$ then
    $
        \sum_{k=0}^{N-1} \GenGap(\nextu, \realoptu)
        \le B^0(\realoptu, u^0),
    $
    ($N \in \N$).
    Since $\GenGap(\nextu, \realoptu) \ge 0$, this shows that $\GenGap(u^N, \realoptu) \to 0$.
    The strong convergence of the primal variable for quadratically minorised $\GenGap$ is then immediate whereas following by Jensen's inequality gives the ergodic convergence claim.
\end{proof}

\begin{example}
    In \cref{sec:convergence:second}, we can take $Pu=\sqrt{\gamma_F-\tilde\gamma_F} x$  if $\gamma_F>\tilde\gamma_F$ or $Pu=\sqrt{\gamma_{G_*}-\tilde\gamma_{G_*}} y$ if $\gamma_{G_*} > \tilde \gamma_{G_*}$. The examples of \cref{sec:convergence:second:block} for $x=(x_1,\ldots,x_m)$, $y=(y_1,\ldots,y_n)$ may allow $Pu=\sqrt{\gamma_{F_j}-\tilde\gamma_{F_j}}x_j$ or $Pu=\sqrt{\gamma_{G_{\ell*}}-\tilde\gamma_{G_{\ell*}}}y_\ell$.
\end{example}

\begin{remark}
    Under similar conditions as  \cref{thm:convergence:strong}, it is possible to obtain $O(1/N^2)$ convergence rates; see \cite{chambolle2010first,tuomov-proxtest} for the convex-concave case and \cite{tuomov-nlpdhgm-redo,tuomov-nlpdhgm-general} in general.
\end{remark}

\begin{theorem}[Linear convergence]\index{convergence!linear}
    \label{thm:convergence:linear}
    Let $F$ and $G_*$ be convex, proper, and lower semicontinuous; $K \in C^1(X \times Y)$; and both $J_X \in C(X)$ and $J_Y \in C(Y)$ convex and Gâteaux-differentiable.
    Suppose for some $\gamma>0$ and $\realoptu \in \inv H(0)$ that
    \begin{enumerate}[label=(\roman*)]
        \item\label{item:convergence:linear:gap} \eqref{eq:bregpdps:condition2} holds with $\GenGap(u,\realoptu) \ge \gamma B^0(\realoptu, u)$ within $\Omega_\realoptu \subset X \times Y$; and
        \item  $B^0$ is elliptic within $\Omega \supset \realoptu$.
    \end{enumerate}
    Let $\{\nextu\}_{k \in \N}$ be generated by the PDBS \eqref{eq:bregpdps:alg:expanded} for any initial $u^0$, and suppose $\{\thisu\}_{k \in \N} \subset \Omega \isect \Omega_\realoptu$.
    Then $B^0(\realoptu, u^N) \to 0$ and $u^N \to \realoptu$ at a linear rate.

    In particular, if  $\GenGap(u, \realoptu) \ge \gamma \norm{u-\realoptu}^2$, ($k \in \N$), for some $\gamma>0$, and $J^0$ is Lipschitz-continuously differentiable, then $u^N \to \realoptu$ at a linear rate.
\end{theorem}

\begin{proof}
    \Cref{lemma:bregpdps:convergence} establishes the quantitative $\Delta$-Féjer monotonicity \eqref{eq:abstractbreg:fejer}. Using \ref{item:convergence:linear:gap}, this yields $(1+\gamma)B^0(\realoptu, \nextu) \le B^0(\realoptu, \thisu)$.
    By the semi-ellipticity of $B^0$, the claimed linear convergence of $B^0(\realoptu, u^N) \to 0$ follows. Since $B^0$ is assumed elliptic, also $u^N \to \realoptu$ linearly.
    If $J^0$ is Lipschitz-continuously differentiable, then, similarly to \cref{lemma:bregpdps:distkplus1-for-k-lowerbound},  $B^0(\realoptu,\nextu) \le L_{DJ}\norm{\nextu-\realoptu}^2$ for some $L_{DJ}>0$. Thus $\GenGap(\nextu, \realoptu) \ge \gamma\inv L_{DJ}B^0(\realoptu,\nextu)$, so the main claim establishes the particular claim.
\end{proof}

\begin{example}
    $J^0$ is Lipschitz-continuously differentiable if $X$ and $Y$ are Hilbert spaces with $J_X=\inv\tau N_X$ and $J_Y=\inv\sigma N_Y$, and $K$ Lipschitz-continuously differentiable.
\end{example}

\subsection{Convergence of gaps in the convex-concave setting}
\label{sec:convergence:gaps}

We finish this section by studying the convergence of gap functionals in the convex-concave setting.

\begin{lemma}\index{convergence!gap}
    \label{lemma:bregpdps:unaccelerated-gap}
    Suppose $F$ and $G_*$ are convex, proper, and lower semicontinuous, and $K \in C^1(X \times Y)$ is convex-concave on $\dom F \times \dom G_*$. Then \eqref{eq:bregpdps:condition2} holds for all $\optu \in X \times Y$ with $\Omega_\optu=X\times Y$ and $\GenGap=\GenGap^\lagrangian$ the \term[gap!Lagrangian]{Lagrangian gap}
    \[
        \begin{aligned}[t]
            \GenGap^\lagrangian(u, \optu)
            & \defeq
            \lagrangian(x, \opty) - \lagrangian(\optx, y)
            \\
            &
            =
            [F(x)+K(x,\opty)-G_*(\opty)]
            -
            [F(\optx)+K(\realoptx, y)-G_*(y)].
        \end{aligned}
    \]
    This functional is non-negative if $\optu \in \inv H(0)$.

    Moreover, if $\sum_{k=0}^{N-1} \GenGap^\lagrangian(\nextu, \optu) \le M(\optu)$ for some $M(\optu) \ge 0$, for all $\optu \in X \times Y$ and all $N \in \N$, and we define the \term{ergodic sequence} $\tilde u^N \defeq \frac{1}{N}\sum_{k=0}^{N-1} \nextu$, then
    \begin{enumerate}[label=(\roman*)]
        \item\label{item:bregpdps:unaccelerated-gap:gapsum}
            $0 \le \frac{1}{N}\sum_{k=0}^{N-1} \GenGap^\lagrangian(\nextu, \realoptu) \to 0$ at the rate $O(1/N)$ for $\realoptu \in \inv H(0)$.
        \item\label{item:bregpdps:unaccelerated-gap:erggap}
            $0 \le \GenGap^\lagrangian(\tilde u^N, \realoptu) \to 0$ at the rate $O(1/N)$ for $\realoptu \in \inv H(0)$.
        \item\label{item:bregpdps:unaccelerated-gap:partgap}
            If $M \in C(X \times Y)$ and $\Omega \subset X \times Y$ is bounded with $\Omega \isect \inv H(0) \ne \emptyset$, then $0 \le \GenGap_\Omega(\tilde u^N) \to 0$ at the rate $O(1/N)$ for the \term[gap!partial]{partial gap}
            \(
                \GenGap_\Omega(u) \defeq  \sup_{\optu \in \Omega} \GenGap^\lagrangian(u, \optu).
            \)
    \end{enumerate}
\end{lemma}

The convergence results in \cref{lemma:bregpdps:unaccelerated-gap} are \term{ergodic} because they apply to sequences of running averages.
To understand the partial gap, we recall that with $K(x,y)=\dualprod{Ax}{y}$ bilinear Fenchel--Rockafellar's theorem show that the \term[gap!duality]{duality gap} $\GenGap^D(u) \defeq [F(x) + G_*(Ax)] + [F_*(-A^*y) + G_*^*(y)] \ge 0 $ and is zero \emph{if and only if} $u \in \inv H(0)$.
The duality gap can be written $\GenGap^D(u) = \GenGap_{X \times Y}(u)$.

\begin{proof}
    By the convex-concavity of $K$ and the definition of the subdifferential,
    \begin{equation*}
        \begin{aligned}[t]
            \dualprod{D_x &K(x,y)}{x-\optx}
            -\dualprod{D_y K(x,y)}{y-\opty}
            \\
            &
            \ge
            [K(x, y)
            - K(\optx, y)]
            -[K(x,y)
            - K(x, \opty)]
            =
            K(x,\opty)-K(\optx, y).
        \end{aligned}
    \end{equation*}
    for all $(x, y) \in X \times Y$. Also using $x^* \in \subdiff F(\nextx)$ and $y^* \in \subdiff G_(\nexty)$ with the definition of the convex subdifferential, we see that $\GenGap=\GenGap^\lagrangian$ satisfies \eqref{eq:bregpdps:condition2}.
    The non-negativity of $\GenGap(\freevar, \realoptu)$ follows by similar reasoning, first using that
    \begin{equation}
        \label{eq:bregpdps:gap-kineq-2}
        K(x, \realopty)
        -K(\realoptx, y)
        \ge
        \dualprod{D_x K(\realoptx, \realopty)}{x-\realoptx}
        -\dualprod{D_y K(\realoptx, \realopty)}{y-\realopty}
    \end{equation}
    for all $(x, y) \in X \times Y$, and following by the definition of the subdifferential applied to $-D_x K(\realoptx, \realopty) \in \subdiff F(\realoptx)$ and $D_y K(\realoptx, \realopty) \in \subdiff G_*(\realopty)$.

    For \ref{item:bregpdps:unaccelerated-gap:gapsum}--\ref{item:bregpdps:unaccelerated-gap:partgap}, we first observe that the semi-ellipticity of $B^0$ and \eqref{eq:bregpdps:condition2} imply $\sum_{k=0}^{N-1} \GenGap^\lagrangian(\nextu, \optu) \le M(\optu)$. Dividing by $N$ and using that $\GenGap^\lagrangian(\nextu, \realoptu) \ge 0$ for $\optu \in \inv H(0)$, we obtain \ref{item:bregpdps:unaccelerated-gap:gapsum}. Jensen's inequality then gives $\GenGap^\lagrangian(\nexxt{\tilde u}, \optu) \le M(\optu)/N$, hence \ref{item:bregpdps:unaccelerated-gap:erggap} for $\optu \in \inv H(0)$. Finally, taking the supremum over $\optu \in \Omega$ gives \ref{item:bregpdps:unaccelerated-gap:partgap} because $M$ is bounded on bounded sets.
\end{proof}

In the following theorem, we may in particular take $K(x,y)=\dualprod{Ax}{y}$ bilinear, or $K(x,y)=\dualprod{Ax}{y}+E(x)$ with $E$ convex. \Cref{lemma:bregpdps:distkplus1-for-k-lowerbound,ex:bregpdps:bk0pos-nonconvex-e,ex:bregpdps:bk0pos-linear} provide step length conditions that ensure the semi-ellipticity required of $B^0$ in \cref{thm:convergence:gap}.

\begin{theorem}[Gap convergence]\index{convergence!gap}
    \label{thm:convergence:gap}
    Let $F: X \to \extR$ and $G_*: Y \to \extR$ be convex, proper, and lower semicontinuous. Also let $K \in C^1(X \times Y)$ be convex-concave within $\dom F \times \dom G_*$. Finally, let $J_X \in C^1(X)$ and $J_Y \in C^1(Y)$ convex.
    If $B^0$ is semi-elliptic, then the iterates $\{\nextu\}_{k \in \N}$ generated by the PDBS \eqref{eq:bregpdps:alg:expanded} for any initial $u^0 \in X \times Y$ satisfy \cref{lemma:bregpdps:unaccelerated-gap}\,\ref{item:bregpdps:unaccelerated-gap:gapsum}--\ref{item:bregpdps:unaccelerated-gap:partgap}.
\end{theorem}

\begin{proof}
    By \cref{lemma:bregpdps:unaccelerated-gap}, holds with $\GenGap=\GenGap^\lagrangian$
    Hence by \cref{lemma:bregpdps:convergence}, \eqref{eq:abstractbreg:result} holds.
    Since $B^0$ is semi-elliptic, this implies that that $\sum_{k=0}^{N-1} \GenGap(\nextu, \optu) \le M(\optu) \defeq B^0(\optu, u^0)$ for all $N \in \N$.
    Since $J_X, J_Y$, and $K$ are continuously differentiable, $M \in C^1(X \times Y)$.
    The rest follows from the second part of \cref{lemma:bregpdps:unaccelerated-gap}.
\end{proof}

\section{Inertial terms}
\label{sec:inertial}

We now generalise \eqref{eq:abstractbreg:pp}, making the involved Bregman divergences dependent on the iteration $k$ and earlier iterates:\index{inertia}
\begin{equation}
    \label{eq:inertiabreg:pp}
    \tag{IPP}
    0 \in H(\nextu) + D_1 B_{k+1}(\nextu,\thisu) + D_1 B_{k+1}^-(\thisu,\prevu),
\end{equation}
for $B_{k+1} \defeq B_{J_{k+1}}$ and $B_{k+1}^- \defeq B_{J_{k+1}^-}$ generated by $J_{k+1}, J_{k+1}^-: U \to \R$.
We take $u^{-1} \defeq u^0$ for this to be meaningful for $k=0$.
Our main reason for introducing the dependence on $\prevu$ is improve \cref{eq:bregpdps:alg:expanded,eq:bregpdps:alg:hilbert} to be explicit in $K$ when $K$ is not affine in $y$: otherwise the dual step of those methods is in general not practical to compute unlike the affine case of \cref{rem:bregpdps:nonlin}.
Along the way we also construct a more conventional inertial method.

\subsection{A generalisation of the fundamental theorem}

We realign indices to get a simple fundamental condition to verify on each iteration:

\begin{theorem}
    \label{thm:inertiabreg}
    On a Banach space $U$, let $H: U \setto U^*$, and let $J_k,J_k^-: U \to \extR$ be Gâteaux-differentiable with the corresponding Bregman divergences $B_k \defeq B_{J_k}$ and $B_k^- \defeq B_{J_k^-}$ for all $k=1,\ldots,N$.
    Suppose \eqref{eq:inertiabreg:pp} is solvable for $\{\nextu\}_{k \in \N}$ given an initial iterate $u^0 \in U$. If for all $k=0,\ldots,N-1$, for some $\optu \in U$ and $\GenGap(\nextu, \optu) \in \R$, for all $\nexxt h \in H(\nextu)$ the modified \term{fundamental condition}
    \begin{equation}
        \label{eq:inertiabreg:condition}
        \tag{IC}
        \dualprod{\nexxt{h}}{\nextu-\optu}
        \ge
        [(B_{k+2}+B_{k+3}^-) - (B_{k+1}+B_{k+2}^-)](\optu, \nextu)
        + \GenGap(\nextu, \optu)
    \end{equation}
    holds, and $B_{k+1}^-$ satisfies the \term{general Cauchy inequality}
    \begin{equation}
        \label{eq:inertiabreg:cauchy}
        \dualprod{D_1 B_{k+1}^-(\thisu, u)}{\thisu-u'}
        \le B_{k+1}'(\thisu, u) + B_{k+1}''(u', \thisu)
        \quad (u,u' \in X)
    \end{equation}
    for some $B_{k+1}',B_{k+1}'': U \times U \to \R$,
    then we have the modified \term{descent inequality}
    \begin{multline}
        \label{eq:inertiabreg:result}
        \tag{ID}
        [B_{N+1}+B_{N+2}^--B_{N+1}''](\optu, u^N)
        + \sum_{k=0}^{N-1} [B_{k+1}+B_{k+2}^- - B_{k+1}'' -B_{k+2}'](\nextu,\thisu)
        \\
        + \sum_{k=0}^{N-1} \GenGap(\nextu, \optu)
        \le
        [B_1+B_2^-](\optu, u^0).
    \end{multline}
\end{theorem}

\begin{proof}
    We can write \eqref{eq:inertiabreg:pp} as
    \begin{equation}
        \label{eq:inertiabreg:pp-expand}
        0=\nexxt h + D_1 B_{k+1}(\nextu, \thisu) + D_1 B_{k+1}^-(\thisu, \prevu)
        \ \ \text{for some}\ \ %
        \nexxt h \in H(\nextu).
    \end{equation}
    Testing \eqref{eq:inertiabreg:pp} by applying $\dualprod{\freevar}{\nextu-\optu}$ we obtain
    \begin{equation*}
        0 = \dualprod{\nexxt{h} + D_1 B_{k+1}(\nextu, \thisu)+D_1 B_{k+1}^-(\thisu, \prevu)}{\nextu-\optu}.
    \end{equation*}
    Summing over $k=0,\ldots,N-1$ and using $u^{-1}=u^0$ to eliminate $B_1^-(u^0, u^{-1})=0$, we rearrange
    \begin{gather}
        \label{eq:inertiabreg:est0}
        0 = S_N + \sum_{k=0}^{N-1} \dualprod{\nexxt{h} + D_1[B_{k+1}+B_{k+2}^-](\nextu, \thisu)}{\nextu-\optu}
        \shortintertext{for}
        \notag
        S_N 
        \defeq
        \dualprod{D_1 B_{J_{N+1}^-}(u^N,u^{N-1})}{\optu-u^N}
        +\sum_{k=0}^{N-1}
            \dualprod{D_1 B_{J_{k+1}^-}(\thisu,\prevu)}{\nextu-\thisu}.
    \end{gather}
    Abbreviating $\bar B_{k+1} \defeq B_{k+1}+B_{k+2}^-$ and using  \eqref{eq:inertiabreg:condition} and the three-point identity \eqref{eq:bregman:three-point-identity} in \eqref{eq:inertiabreg:est0} we obtain
    \[
        0 \ge S_N
        +
        \sum_{k=0}^{N-1}\left(
        \bar B_{k+2}(\optu, \nextu)
        -\bar B_{k+1}(\optu, \thisu)
        + \bar B_{k+1}(\nextu, \thisu)
        + \GenGap(\nextu, \optu)
        \right).
    \]
    Using the generalised Cauchy inequality \eqref{eq:inertiabreg:cauchy} and, again, that  $u^{-1}=u^0$, we get
    \[
        \begin{aligned}[t]
        S_N &
        \ge
        -B_{N+1}'(u^N,u^{N-1})-B_{N+1}''(\optu, u^N)
        -\sum_{k=0}^{N-1}
        \left(
            B_{k+1}'(\thisu,\prevu)
            +B_{k+1}''(\nextu,\thisu)
        \right)
        \\
        &
        =
        -B_{N+1}''(\optu, u^N)
        -\sum_{k=0}^{N-1} [B_{k+1}''+B_{k+2}'](\nextu,\thisu).
        \end{aligned}
    \]
    These two inequalities yield \eqref{eq:inertiabreg:result}.
\end{proof}

\subsection{Inertia (almost) as usually understood}
\label{sec:inertiabreg:inertia}

We take $J_{k+1}=J^0$ and $J_{k+1}^-=-\lambda_k J^0$ for some $\lambda_k \in \R$.
We then expand \eqref{eq:inertiabreg:pp} as
\begin{alg}{Inertial PDBS}%
    \index{primal-dual!splitting!inertial}%
    Iteratively over $k \in \N$, solve for $\nextx$ and $\nexty$:
    \begin{equation}
        \label{eq:bregpdps-inertia:alg}
            \begin{aligned}
                (1+\lambda_k)& [DJ_X(\thisx)-D_x K(\thisx, \thisy)]-\lambda_k [DJ_X(\prevx)-D_xK(\prevx,\prevy)]
                \\ & \in DJ_X(\nextx) + \subdiff F(\nextx),
                \\
                (1+\lambda_k)&[DJ_Y(\thisy)-D_y K(\thisx, \thisy)]-\lambda_k[DJ_Y(\prevy)-D_yK(\prevx,\prevy)]
                \\ & \in DJ_Y(\nexty) + \subdiff G_*(\nexty)-2D_y K(\nextx, \nexty)
            \end{aligned}
    \end{equation}
\end{alg}
If $X$ and $Y$ are Hilbert spaces with $J_X=\inv\tau N_X$ and $J_Y=\inv\sigma N_Y$ the standard generating functions divided by some step length parameters $\tau,\sigma>0$, and $K(x,y)=\dualprod{Ax}{y}$ for $A \in \linear(X; Y)$, \eqref{eq:bregpdps-inertia:alg} reduces to the inertial method of \cite{chambolle2014ergodic}:
\begin{alg}{Inertial PDPS for bilinear $K$}%
    With initial $\tilde x^0=x^0$ and $\tilde y^0=y^0$, iterate over $k \in \N$:
    \begin{equation}
        \label{eq:bregpdps-inertia:hilbert-linear}
            \begin{aligned}
                \nextx & \defeq \prox_{\tau F}(\this{\tilde x} - \tau A^*\this{\tilde y}), \\
                \nexty & \defeq \prox_{\sigma G_*}(\this{\tilde y} + \sigma A(2\nextx-\this{\tilde x})), \\
                \nexxt{\tilde x} & \defeq (1+\lambda_{k+1})\nextx-\lambda_{k+1}\thisx, \\
                \nexxt{\tilde y} & \defeq (1+\lambda_{k+1})\nexty-\lambda_{k+1}\thisy. \\
            \end{aligned}
    \end{equation}
\end{alg}
More generally, however, \eqref{eq:bregpdps-inertia:alg} does not directly apply inertia to the iterates. It applies inertia to $K$.

The general Cauchy inequality \eqref{eq:inertiabreg:cauchy} automatically holds by the three-point identity \eqref{eq:bregman:three-point-identity} with $J_{k+1}''=J_{k+1}'=J_{k+1}^-$ if $B_{k+1}^- \ge 0$, which is to say that $J_{k+1}^-$ is convex. This is the case if $\lambda_k \le 0$. For usual inertia we, however, want $\lambda_k > 0$.
We will therefore use  \cref{lemma:bregman:cauchy}, requiring:

\begin{assumption}
    \label{ass:bregpdps-inertia:main}
    For some $\beta>0$, in a domain $\Omega \subset X \times Y$,
    \begin{equation}
        \label{eq:bregpdps-inertia:cauchy-ass}
        \abs{\dualprod{D_1 B^0(\thisu, u)}{\thisu-u}}
        \le B^0(\thisu, u) + \beta B^0(u', \thisu)
        \quad (u, u', \thisu \in \Omega).
    \end{equation}
    Moreover, the parameters $\{\lambda_k\}_{k \in \N}$ are non-increasing and for some $\epsilon>0$,
    \begin{equation}
        \label{eq:bregpdps-inertia:lambda-bound}
        0 \le \lambda_{k+1} \le \frac{1-\epsilon-\lambda_k \beta}{2}
        \quad (k \in \N).
    \end{equation}
\end{assumption}

\begin{example}
    \label{ex:bregpdps-inertia:main:0}
    Suppose the generating function $J^0$ is $\gamma$-strongly subdifferentiable (i.e., $B^0$ is $\gamma$-elliptic, see \cref{sec:convergence:ellipticity,sec:convergence:ellipticity:block}) within $\Omega \subset X \times Y$ and satisfies the subdifferential smoothness property \eqref{eq:smoothness-grad-smoothness} with the factor $L>0$. Then by \cref{lemma:bregman:cauchy}, \eqref{eq:bregpdps-inertia:cauchy-ass} holds with $\beta=L\inv\gamma$ in some domain $\Omega \subset X \times Y$.

    As a particular case, let $X$ and $Y$ be Hilbert spaces with the standard generating functions $J_X=\inv\tau N_X$, $J_Y=\inv\sigma N_Y$. Also let $DK$ be $L_{DK}$-Lipschitz within $\Omega$. Then $J^0$ is Lipschitz with factor $L=\max\{\inv\sigma,\inv\tau\}+L_{DK}$.
    Consequently  the required subdifferential smoothness property \eqref{eq:smoothness-grad-smoothness} holds with the same factor $L$; see \cite[Theorem 18.15]{bauschke2017convex} or \cite[Appendix C]{tuomov-proxtest}.
    
    We computed $L_{DK}$ for some specific $K$ in \cref{sec:convergence:ellipticity}.
\end{example}

\begin{example}
    \label{ex:bregpdps-inertia:main:1}
    If $K(x,y)=\dualprod{Ax}{y}$ with $A \in \linear(X; Y^*)$, and if $J_X=\inv\tau N_X$, $J_Y=\inv\sigma N_Y$,in Hilbert spaces $X$ and $Y$, then $B^0(u', u)=\tfrac{1}{2\tau}\norm{x-x'}^2+\tfrac{1}{2\sigma}\norm{y-y'}^2+\dualprod{A(x-x')}{y-y'}$. By standard Cauchy inequality, \eqref{eq:bregpdps-inertia:cauchy-ass} holds for $\beta=1$ in $\Omega=X \times Y$.
    Consequently the next example recovers the upper bound for $\lambda$ in \cite{chambolle2014ergodic}:
\end{example}

\begin{example}
    \label{ex:bregpdps-inertia:main:2}
    The bound \eqref{eq:bregpdps-inertia:lambda-bound} holds for some $\epsilon>0$ if $\lambda_k \equiv \lambda$ for $0 \le \lambda < 1/(2+\beta)$.
\end{example}

\begin{lemma}
    \label{lemma:bregpdps-inertia:mainlemma}
    Suppose \cref{ass:bregpdps-inertia:main} holds and that \eqref{eq:bregpdps:condition2} holds within $\Omega_\optu$ for some $\optu \in \Omega$ and $\GenGap(u, \optu)$. Given $u^0 \in \Omega$, suppose the iterates generated by the inertial PDBS \eqref{eq:bregpdps-inertia:alg} satisfy $\{\thisu\}_{k=0}^N \subset \Omega_\optu \isect \Omega$. Then
    \begin{equation}
        \label{eq:inertiabreg:inertia:result}
        \epsilon B^0(\optu, u^N)
        + \epsilon \sum_{k=0}^{N-1} B^0(\nextu,\thisu)
        + \sum_{k=0}^{N-1} \GenGap(\nextu, \optu)
        \le
        (1-\lambda_1)B^0(\optu, u^0).
    \end{equation}
\end{lemma}

\begin{proof}
    Since $B_{k+1}=B^0$ and $B_{k+1}^-=-\lambda_k B^0$ for all $k \in \N$,
    \[
        (B_{k+2}+B_{k+3}^-) - (B_{k+1}+B_{k+2}^-)
        =(\lambda_{k+1}-\lambda_{k+2})B^0.
    \]
    Since $\lambda_k$ is decreasing and $B^0$ is semi-elliptic within $\Omega \supset \{\thisu,\optu\}$, we deduce that $(\lambda_{k+1}-\lambda_{k+2})B^0(\optu,\thisu) \ge 0$. Consequently \eqref{eq:inertiabreg:condition} holds if \eqref{eq:abstractbreg:condition} does. By the proof of \cref{lemma:bregpdps:convergence}, \eqref{eq:inertiabreg:condition} then holds if \eqref{eq:bregpdps:condition2} does.
    Using  \eqref{eq:bregpdps-inertia:cauchy-ass}, \eqref{eq:inertiabreg:cauchy} holds with $B_{k+1}'=\lambda_k B_0$ and $B_{k+1}''=\lambda_k \beta B_0$.
    Referring to \cref{thm:inertiabreg}, we now obtain \eqref{eq:inertiabreg:result}.
    We expand
    \begin{align*}
        [B_{N+1} + B_{N+2}^- - B_{N+1}''](\optu, u^N)
        &
        =
        (1-\lambda_{k+1}-\lambda_k\beta)B^0(\optu, u^N)
        \quad\text{and}
        \\
        [B_{k+1} + B_{k+2}^- -B_{k+1}'' -B_{k+2}'](\nextu, \thisu)
        &
        =
        (1-\lambda_{k+1} -\lambda_k\beta -\lambda_{k+1})B^0(\nextu, \thisu).
    \end{align*}
    Since $\optu, u^k \in \Omega$ for all $k=0,\ldots,N$, using the ellipticity of $B^0$ within $\Omega$ as well as \eqref{eq:bregpdps-inertia:lambda-bound} we now estimate the first from below by
    $\epsilon B^0(\optu, u^N)$ and the second by $\epsilon B^0(\nextu, \thisu)$.
    Thus \eqref{eq:inertiabreg:result} produces \eqref{eq:inertiabreg:inertia:result}.
\end{proof}

We may now proceed as in \cref{sec:convergence:gaps,sec:convergence:iterate} to prove convergence.
For the verification of \cref{ass:bregpdps-inertia:main} we can use \cref{ex:bregpdps-inertia:main:0,ex:bregpdps-inertia:main:1,ex:bregpdps-inertia:main:2}.

\begin{theorem}[Convergence, inertial method]
    \index{convergence!gap}
    \index{convergence!weak}
    \index{convergence!strong}
    \label{thm:bregpdps-inertia:convergence}
    \Cref{thm:convergence:gap,thm:convergence:weak,thm:convergence:strong} apply to the iterates $\{\nextu\}_{k \in \N}$ generated by the inertial PDBS \eqref{eq:bregpdps-inertia:alg} if we replace the assumptions of \mbox{(semi-)}ellipticity of $B^0$ with \cref{ass:bregpdps-inertia:main}.
\end{theorem}

\begin{proof}
    We replace \cref{lemma:bregpdps:convergence} and \eqref{eq:abstractbreg:result} by \cref{lemma:bregpdps-inertia:mainlemma} and \eqref{eq:inertiabreg:inertia:result} in the proofs of \cref{thm:convergence:gap,thm:convergence:weak,thm:convergence:strong}.
    Observe that \cref{ass:bregpdps-inertia:main} implies that $B^0$ is \mbox{(semi-)}elliptic.
\end{proof}

\begin{remark}
    The inertial PDPS is improved in \cite{tuomov-inertia} to yield \index{ergodic}\termnoindex{non-ergodic} convergence of the Lagrangian gap. To do the “inertial unrolling” that leads to such estimates, one, however, needs to correct for the anti-symmetry introduced by $K$ into $H$.
\end{remark}

\begin{remark}
    \label{rem:bregpdps-inertia:linear}
    \index{testing}
    \index{Féjer-monotonicity}
    Since \cref{thm:inertiabreg} does not provide the quantitative $\Delta$-Féjer monotonicity used in \cref{thm:convergence:linear}, we cannot prove linear convergence using our present simplified “testing” approach lacking the “testing parameters” of \cite{tuomov-proxtest}.
\end{remark}

\subsection{Improvements to the basic method without dual affinity}
\label{sec:bregpdps-inertia}

We now have the tools to improve the basic PDBS \eqref{eq:bregpdps:alg:expanded} to enjoy prox-simple steps for general $K$ not affine in $y$. Compared to \eqref{eq:bregpdps:j0} we amend $J_{k+1}=J^0$ by taking
\begin{equation}
    \label{eq:bregpdps-inertia:jk-final}
    \begin{aligned}[t]
    J_{k+1}(x, y)
    &
    \defeq
    J_X(x) + J_Y(y) - K(x, y) + 2K(\nextx, y)
    \\
    &
    = J^0(x, y) + 2 K(\nextx, y).
    \end{aligned}
\end{equation}
This would be enough for $K$ to be explicit in the algorithm, however, proofs of convergence would practically require $G_*$ to be strongly convex even in the convex-concave case.
To fix this, we introduce the inertial term generated by
\begin{equation}
    \label{eq:bregman:jkm-final}
    J_{k+1}^-(u) \defeq [J^0 - J_k](u) = -2K(\thisx, y).
\end{equation}
As always, we write $B_{k+1}$, $B^0$, and $B_{k+1}^-$ for the Bregman divergences generated by $J_{k+1}$, $J^0$, and $J_{k+1}^-$.

Since
\begin{gather*}
    D_1[B_{k+1}-B^0](\thisu, \prevu)+D_1 B_{k+1}^-(\thisu, \prevu)
    = (0, \tilde y_{k+1}^*)
    \shortintertext{for}
    \tilde y_{k+1}^*
    =
    2[
        D_yK(\nextx, \nexty)-D_yK(\nextx,\thisy)
        -D_yK(\thisx, \thisy)+D_yK(\thisx,\prevy)
    ],
\end{gather*}
the algorithm \eqref{eq:inertiabreg:pp} expands similarly to \eqref{eq:bregpdps:alg:expanded} as the
\begin{alg}{Modified PDBS}%
    \index{primal-dual!splitting!Bregman-proximal}%
    \index{primal-dual!splitting!Bregman-proximal!modified}%
    \index{PDBS!modified}%
    Iteratively over $k \in \N$, solve for $\nextx$ and $\nexty$:
    \begin{equation}
        \label{eq:bregpdps-inertia:alg:expanded}
            \begin{aligned}
                DJ_X(\thisx)- D_x K(\thisx, \thisy) &
                \in DJ_X(\nextx) + \subdiff F(\nextx)
                \quad\text{and}
                \\
                \multispan2{$\displaystyle
                DJ_Y(\thisy)
                +  [2D_y K(\nextx, \thisy)+ D_y K(\thisx, \thisy) - 2D_y(\thisx,\prevy)]$}
                \\
                &
                \in
                DJ_Y(\nexty) + \subdiff G_*(\nexty).
                \end{aligned}
    \end{equation}
\end{alg}
The method reduces to the basic PDBS \eqref{eq:bregpdps:alg:expanded} when $K$ is affine in $y$.
In Hilbert spaces $X$ and $Y$ with $J_X=\inv\tau N_X$ and $J_Y=\inv\sigma N_Y$, we can rearrange \eqref{eq:bregpdps-inertia:alg:expanded} as
\begin{alg}{Modified PDPS}%
    \index{primal-dual!splitting!proximal}%
    \index{primal-dual!splitting!proximal!modified}%
    \index{PDPS!modified}%
    Iterate over $k \in \N$:
    \begin{equation}
        \label{eq:bregpdps-inertia:hilbert}
        \!\!\!\!%
            \begin{aligned}
                \nextx & \defeq \prox_{\tau F}(\thisx - \tau \grad_x K(\thisx, \thisy)),
                \\
                \nexty & \defeq \prox_{\sigma G_*}(\thisy + \sigma[2\grad_y K(\nextx, \thisy) + \grad_y K(\thisx, \thisy)-2\grad_yK(\thisx,\prevy)]).\!\!\!\!%
            \end{aligned}
    \end{equation}
\end{alg}

\begin{remark}
    \label{rem:bregpdps-inertia:nonlin}
    The modified PDPS  \eqref{eq:bregpdps-inertia:hilbert} is slightly more complicated than the method in \cite{tuomov-nlpdhgm-general}, which would update
    \[
        \nexty \defeq \prox_{\sigma G_*}(\thisy + \sigma \grad_y K(2\nextx-\thisx, \thisy)).
    \]
    Likewise, \eqref{eq:bregpdps-inertia:alg:expanded} is different from the algorithm presented in \cite{hamedani2018primal} for convex-concave $K$. It would, for the standard generating functions, update\footnote{Note that \cite{hamedani2018primal} uses the historical ordering of the primal and dual updates from \cite{chambolle2010first}, prior to the proof-simplifying discovery of the proximal point formulation in \cite{he2012convergence}. Hence our $\thisy$ is their $\nexty$.}
    \[
        \nexty \defeq \prox_{\sigma G_*}(\thisy + \sigma[2\grad_y K(\nextx, \thisy)-\grad_y K(\thisx, \prevy)]).
    \]
    We could produce this method by taking $J_{k+1}^-(u) = -K(\thisx, y)$. However, the convergence proofs would require some additional steps.
\end{remark}

The main difference to the overall analysis of \cref{sec:convergence} is in bounding from below the Bregman divergences in \cref{eq:inertiabreg:result}.
We now have
\begin{subequations}
\label{eq:bregpdps-modified:divergence-sums}
\begin{align}
    B_{N+1} + B_{N+2}^- - B_{N+1}''
    &=B^0 - B_{N+1}''
    \quad\text{and}
    \\
    B_{k+1}+B_{k+2}^- - B_{k+1}'' -B_{k+2}'
    &=B^0 - B_{k+1}'' -B_{k+2}'.
\end{align}
\end{subequations}
If $D_y K(\thisx, \freevar)$ is $L_{DK,y}$-Lipschitz,
\begin{equation}
    \label{eq:bregpdps-modified:cauchy}
    \begin{aligned}[t]
    \dualprod{D_1 B_{k+1}^-(\thisu, u)}{\thisu-u'}
    &=2\dualprod{D_y K(\thisx, \thisy)-D_y K(\thisx, y)}{\thisy-y'}
    \\
    &
    \le \sqrt{L_{DK,y}}\norm{y-\thisy}^2 + \sqrt{L_{DK,y}} \norm{y'-\thisy}^2
    \\
    &
    =: B_{k+1}'(\thisu, u) + B_{k+1}''(u',\thisu).
    \end{aligned}
\end{equation}
Therefore, for the modified descent inequality \cref{eq:inertiabreg:result} to be meaningful, we require:

\begin{assumption}
    \label{ass:bregpdps-inertia:b0}
    We assume that $\norm{D_y K(x, y)-D_y K(x, y')} \le L_{DK,y}\norm{y-y'}$ when $(x, y), (x, y') \in \Omega$ for some domain $\Omega \subset X \times Y$. Moreover, for some $\epsilon \ge 0$ we have
    \begin{equation}
        \label{eq:bregpdps-inertia:ellipticity}
        B^0(u, u') \ge \frac{\epsilon}{2}\norm{u-u'}_{X \times Y}^2
        +2\sqrt{L_{DK,y}}\norm{y-y'}_Y^2
        \quad (u, u' \in \Omega).
    \end{equation}
    We say that the present assumption holds \emph{strongly} if $\epsilon>0$.
\end{assumption}

\begin{example}
    \label{ex:bregpdps-inertia:b0:0}
    If $K$ is affine in $y$, $L_{DK,y}=0$. Therefore, \cref{ass:bregpdps-inertia:b0} reduces to the \mbox{(semi-)}ellipticity of $B^0$, which can be verified as in \cref{sec:convergence:ellipticity,sec:convergence:ellipticity:block}.
\end{example}

\begin{example}
    \label{ex:bregpdps-inertia:b0:1}
    Generally, it is easy to see that if one of the results of \cref{sec:convergence:ellipticity} holds with $\tilde\sigma=1/(\inv\sigma-4\sqrt{L_{DK,y}})>0$ in place of $\sigma$, then \eqref{eq:bregpdps-inertia:ellipticity} holds. In particular, if $K$ has $L_{DK}$-Lipschitz derivative within $\Omega$, then \cref{lemma:bregpdps:distkplus1-for-k-lowerbound} gives the condition $1 \ge  L_{DK}\max\{\tau,\sigma/(1-4\sigma \sqrt{L_{DK,y}})\}$ and $1>4\sigma \sqrt{L_{DK,y}}$ for \eqref{eq:bregpdps-inertia:ellipticity} to hold with $\epsilon=0$.
    The assumption holds strongly if the first inequality is strict.
\end{example}

Similarly to \cref{lemma:bregpdps-inertia:mainlemma}, we now have the following replacement for  \cref{lemma:bregpdps:convergence}:

\begin{lemma}
    \label{lemma:bregpdps-modified:convergence}
    Suppose \cref{ass:bregpdps-inertia:b0} holds and \eqref{eq:bregpdps:condition2} holds within $\Omega_\optu$ for some $\optu \in X \times Y$ and $\GenGap(u, \optu)$. Given $u^0 \in X \times Y$, suppose the iterates generated by the modified PDBS \eqref{eq:bregpdps-inertia:alg:expanded} satisfy $\{\thisu\}_{k=0}^N \subset \Omega_\optu$.
    Then
    \begin{equation}
        \label{eq:bregpdps-modified-result}
        \epsilon B^0(\optu, u^N)
        + \epsilon \sum_{k=0}^{N-1} B^0(\nextu,\thisu)
        + \sum_{k=0}^{N-1} \GenGap(\nextu, \optu)
        \le
        [B_1+B_2^-](\optu, u^0).
    \end{equation}
\end{lemma}

\begin{proof}
    Inserting \eqref{eq:bregpdps-inertia:jk-final} and \eqref{eq:bregman:jkm-final}, \eqref{eq:inertiabreg:condition} reduces to  \eqref{eq:abstractbreg:condition}, which follows from \eqref{eq:bregpdps:condition2} as in  \cref{lemma:bregpdps:convergence}.
    We verify \eqref{eq:inertiabreg:cauchy} via \eqref{eq:bregpdps-modified:cauchy} and \cref{ass:bregpdps-inertia:b0}.
    Thus \cref{thm:inertiabreg} proves \cref{eq:inertiabreg:result}.
    Inserting \eqref{eq:bregpdps-modified:divergence-sums} and \eqref{eq:bregpdps-inertia:ellipticity} with $B_{k+1}'$ and $B_{k+1}''$ from \eqref{eq:bregpdps-modified:cauchy} into \cref{eq:inertiabreg:result} proves \eqref{eq:bregpdps-modified-result}.
\end{proof}

We may now proceed as in \cref{sec:convergence:gaps,sec:convergence:iterate} to prove convergence.
For the verification of \cref{ass:bregpdps-inertia:b0} we can use \cref{ex:bregpdps-inertia:b0:0,ex:bregpdps-inertia:b0:1}.

\begin{theorem}[Convergence, modified method]
    \index{convergence!gap}
    \index{convergence!weak}
    \index{convergence!strong}
    \label{thm:bregpdps-modified:convergence}
    \Cref{thm:convergence:gap,thm:convergence:weak,thm:convergence:strong} apply to the iterates $\{\nextu\}_{k \in \N}$ generated by the modified PDBS \eqref{eq:bregpdps-inertia:alg:expanded} if we replace the assumptions of semi-ellipticity (resp.~ellipticity) of $B^0$ with \cref{ass:bregpdps-inertia:b0} holding (strongly).
\end{theorem}

\begin{proof}
    We replace \cref{lemma:bregpdps:convergence} and \cref{eq:abstractbreg:result} by \cref{lemma:bregpdps-modified:convergence} and  \eqref{eq:bregpdps-modified-result} in \cref{thm:convergence:gap,thm:convergence:weak,thm:convergence:strong}.
    Observe that (strong) \cref{ass:bregpdps-inertia:b0} implies the \mbox{(semi-)}ellipticity of $B^0$.
\end{proof}

Now we have a locally convergent method \eqref{eq:bregpdps-inertia:hilbert} with easily implementable steps to tackle problems such as Potts segmentation \eqref{eq:intro:potts} \cite{tuomov-nlpdhgm-general}.

\section{Further directions}
\label{sec:further}

We close by briefly reviewing some things not covered, other possible extensions, and alternative algorithms.

\subsection{Acceleration}
\label{sec:accel-brief}

To avoid technical detail, we did not cover $O(1/N^2)$ acceleration.
The fundamental ingredients of proof are, however, exactly the same as we have used: sufficient second-order growth and ellipticity of the Bregman divergences $B^0_k$, which are now iteration-dependent. Additionally, a portion of the second-order growth must be used to make the metrics $B^0_k$ grow as $k \to \infty$.
For bilinear $K$ in Hilbert spaces, such an argument can be found in \cite{tuomov-proxtest};
for $K(x,y)=\dualprod{A(x)}{y}$ in \cite{tuomov-nlpdhgm-redo}; and for general $K$ in \cite{tuomov-nlpdhgm-general}. As mentioned in \cref{rem:bregpdps:nonlin,rem:bregpdps-inertia:nonlin}, the algorithms in the latter two differ slightly from the ones presented here.

\subsection{Stochastic methods}

It is possible to refine the block-adapted \eqref{eq:bregpdps:alg:block} and its accelerated version into \indexalso{stochastic methods}. The idea is to take on each step subsets of primal-blocks $S(i) \subset \{1,\ldots,m\}$ and dual blocks $V(i+1) \subset \{1,\ldots,n\}$ and to only update the corresponding $\nextx_j$ and $\nexty_\ell$. Full discussion of such technical algorithms are outside the scope of our present overview.
We refer to \cite{tuomov-blockcp} for an approach covering block-adapted acceleration and both primal- and dual randomisation in the case of bilinear $K$, but see also \cite{chambolle2017stochastic} for a more basic version. For more general $K$ affine in $y$, see  \cite{tuomov-nlpdhgm-block}.

\subsection{Alternative Bregman divergences}

We have used Bregman divergences as a proof tool, in the end opting for the standard quadratic generating functions on Hilbert spaces. Nevertheless, our theory works for arbitrary Bregman divergences. The practical question is whether $F$ and $G_*$ remain prox-simple with respect to such a divergence. This can be the case for the “entropic distance” generated on $L^1(\Omega; [0, \infty))$ by
\[
    J(x) \defeq
    \begin{cases}
        \int_{\Omega} x(t) \ln x(t) \d t, & x \ge0 \text{ a.e. on } \Omega, \\
        \infty, & \text{otherwise}
    \end{cases}
\]
See, for example, \cite{burger2019entropic} for a Landweber method (gradient descent on regularised least squares) based on such a distance.

\subsection{Alternative approaches}

The derivative $D_1 B^0$ in \eqref{eq:bregpdps:alg} can be seen as a preconditioner, replacing $\tau(u-u')$ in the proximal point method \eqref{eq:bregpdps:basicprox}.
Our choice of $B^0$ is not the only option.

Consider the problem
\begin{equation}
    \label{eq:further:fbprob}
    \min_{x \in X} F(x) + E(x).
\end{equation}
Provided $E$ is differentiable and $F$ \term{prox-simple}, i.e., the proximal map of $F$ has a a closed-form expression, \eqref{eq:intro:fgprob} can be solved by forward-backward splitting methods as first introduced in \cite{lionsmercier1979splitting}. In a Hilbert space $X$, this can be written
\begin{equation}
    \label{eq:further:fb}
    \nextx \defeq \prox_{\tau F}(\thisx - \tau \grad E(\thisx)).
\end{equation}
Variants based on Bregman divergences were introduced in \cite{nemirovski1983problem} under the name “\term{mirror prox}” or “\term{mirror descent}”; see also the review \cite{chambollepock2016introduction}. The method and convergence proofs for it can be derived from our primal-dual approach. Indeed, if we take $G_* \equiv \delta_{\{0\}}$ as the \hyperref[sec:glossary]{indicator function} of zero, and $K(x,y)=E(x)$ for some $E \in C^1(X)$, then \eqref{eq:intro:genprob} is equivalent to \eqref{eq:further:fbprob}.
Now the dual step step of \eqref{eq:bregpdps:alg:hilbert} is $\nexty \defeq 0$, and the primal step is \eqref{eq:further:fb}.

Forward-backward splitting is especially popular under the name \term{iterative soft-thresholding} (ISTA) in the context of \term{sparse reconstruction} (i.e., regularisation of linear inverse problems with $\ell^1$ penalties), see, e.g., \cite{chambolledevore1998nonlinear,daubechies2004surrogate,beck2009fista}.
However, forward-backward splitting has limited applicability in imaging and inverse problems due to the joint prox-simplicity and smoothness requirements.
\index{problem!dual} Sometimes these can be circumvented by considering so-called dual problems \cite{beck2009fast}.


Let then $E$ be Gâteaux-differentiable and $F=G \circ A$ for a nonsmooth function $F$ and a linear operator $A$ in \eqref{eq:further:fbprob}, i.e., consider the problem
\[
    \min_{x \in X} E(x) + G(Ax),
\]

Forward--backward splitting is impractical as $G \circ A$ is in general not \indexalso{prox-simple}. Assuming $G$ to have the preconjugate $G_*$, we can write this problem as an instance of \eqref{eq:intro:genprob} with $F=0$ and $K(x,y)=E(x)+\dualprod{Ax}{y}$. Therefore the methods we have presented are applicable. However, in this instance, also $J^0(u) \defeq \frac{1}{2}\norm{u}_{X \times Y}^2 + \frac{1}{2}\norm{A^*y}^2_{X^*}$ would produce an algorithm with realisable steps. In analogy to the PDPS, it might be called the \term[primal-dual!splitting!explicit]{primal dual explicit spitting} (\indexalso{PDES}). The method was introduced in  \cite{loris2011generalization} for $E(z)=\frac{1}{2}\norm{b-z}^2$ as the \index{iterative soft thresholding!generalised} “generalised iterative soft-thresholding” (\indexalso{GIST}), but has also been called the \term[primal-dual!fixed point method]{primal-dual fixed point method} (\indexalso{PDFP}, \cite{chen2013pdfp}) and the \term{proximal alternating predictor corrector} (\indexalso{PAPC}, \cite{drori2015simple}).


The classical \term{Augmented Lagrangian} method solves the   saddle point problem
\begin{equation}
    \label{eq:further:augl-problem}
    \min_{x}\max_y~ F(x) + \frac{\tau}{2}\norm{E(x)}^2 + \dualprod{E(x)}{y},
\end{equation}
alternatingly for $x$ and $y$. The \term{alternating directions method of multipliers} (\indexalso{ADMM}) of \cite{gabay,arrow1958strudies}  takes $E(x)=Ax_1+Bx_2-c$ and $F(x)=F_1(x_1)+F_2(x_2)$ for $x=(x_1, x_2)$, and alternates between solving \eqref{eq:further:augl-problem} for $x_1$, $x_2$, and $y$, using the most recent iterate for the other variables. The method cannot be expressed in our Bregman divergence framework, as the preconditioner $D_1 B_{k+1}(\freevar, \thisx)$ would need to be non-symmetric. The steps of the method are potentially expensive, each itself being an optimisation problem. Hence the \termnoindex{preconditioned ADMM} of \cite{zhang2011unified}, which is equivalent to the PDPS and the classical \term{Douglas--Rachford splitting} (\indexalso{DRS}, \cite{douglas1956numerical}) applied to appropriate problems \cite{chambolle2010first,clasonvalkonen2020nonsmooth}. The preconditioned ADMM was extended to nonlinear $E$ in \cite{benning2015preconditioned}.


Based on derivations avoiding the Lipschitz gradient assumption (\indexalso{cocoercivity}) in forward-backward splitting, \cite{malitsky2018forward} moves the over-relaxation step $\nexxt{\bar x} \defeq 2\nextx-\thisx$ of the PDPS outside the proximal operators. This amounts to taking $J_{k+1}^-=\lambda_k K$ in \cref{sec:inertiabreg:inertia} instead of $J_{k+1}^-(x,y)=\lambda_k J^0=\lambda_k[\inv\tau J_X(x)+\inv\sigma J_Y(y)-K(x,y)]$, so is “\index{inertia!partial}partial inertia”; compare the “\index{inertia!corrected}corrected inertia” of \cite{tuomov-inertia}.

An \term[over-relaxation]{over-relaxed} variant of the same idea maybe found in \cite{bredies2015accelerated}. We have not discussed over-relaxation of entire algorithms. To briefly relate it to the basic inertia of \cref{eq:bregpdps-inertia:hilbert-linear}, the latter ``rebases'' the algorithm at the inertial iterate $\this{\tilde u}$ constructed from $\thisu$ and $\prevu$, whereas over-relaxation would construct $\this{\tilde u}$ from $\thisu$ and $\prev{\tilde u}$.
The derivation in \cite{bredies2015accelerated} is based on applying Douglas--Rachford splitting on a lifted problem.
The basic over-relaxation of the PDPS is known as the \indexalso{Condat--V{\~u} method} \cite{condat2013primaldual,vu2013splitting}.

\subsection{Functions on manifolds and Hadamard spaces}

The PDPS has been extended in \cite{bergmann2019fenchel} to functions on \index{manifold!Riemannian}Riemannian manifolds; the problem $\min_{x \in \mathcal{M}} F(x)+G(Ex)$, where $E: \mathcal{M} \to \mathcal{N}$ with $\mathcal{M}$ and $\mathcal{N}$ Riemannian manifolds. In general, between manifolds, there are no linear maps, so $E$ is nonlinear. Indeed, besides introducing a theory of conjugacy for functions on manifolds, the algorithm presented in \cite{bergmann2019fenchel} is based on the NL-PDPS of \cite{tuomov-nlpdhgm,tuomov-nlpdhgm-redo}.

Convergence could only be proved on \term[manifold!Hadamard]{Hadamard manifolds}, which are  special: a type of three-point inequality holds \cite[Lemma 12.3.1]{docarmo2013riemannian}.
Indeed, in even more general \term[Hadamard space]{Hadamard spaces} with the metric $d$, for any three points $\nextx,\thisx,\optx$, we have \cite[Corollary 1.2.5]{bacak2014convex}
\begin{equation}
    \label{eq:further:hadamard-three-point}
    \frac{1}{2}d(\thisx, \nextx)^2 + \frac{1}{2}d(\nextx, \optx)^2 - \frac{1}{2}d(\thisx, \optx)^2 \le d(\thisx, \nextx)d(\optx, \nextx).
\end{equation}
Therefore, given a function $f$ on such a space, to derive a simple proximal point algorithm, having constructed the iterate $\thisx$ we might try to find $\nextx$ such that
\[
    f(\nextx) + d(\thisx, \nextx) \le f(\thisx).
\]
Multiplying this inequality by $d(\optx,\nextx)$ and using the three-point inequality \eqref{eq:further:hadamard-three-point},
\[
    \frac{1}{2}d(\thisx, \nextx)^2 + \frac{1}{2}d(\nextx, \optx)^2
    + [f(\nextx)-f(\thisx)]d(\optx,\nextx)
    \le \frac{1}{2}d(\thisx, \optx)^2.
\]
If the space is bounded, $d(\optx, \nextx) \le C$, so since $f(\thisx) \ge f(\nextx)$, we may telescope and proceed as before to obtain convergence.

The Hadamard assumption is restrictive: if a Banach space is Hadamard, it is Hilbert, while a Riemannian manifold is Hadamard if it is simply connected with a non-positive sectional curvature \cite[section 1.2]{bacak2014convex}.

\begin{acknowledgement}
    Academy of Finland grants 314701 and 320022.
\end{acknowledgement}

\section*{Glossary}
\label{sec:glossary}
\addcontentsline{toc}{section}{Glossary and notation}
\begin{list}{}{
    \setlength{\leftmargin}{11.5em}
    \setlength{\labelwidth}{\leftmargin-.5em}
    \setlength{\labelsep}{0.5em}
    \setlength{\itemsep}{1pt}
    \setlength\parsep{1pt}
    \setlength\topsep{1pt}
    \def\makelabel#1{#1\hfill}
}
    \item[The extended reals]
        We define $\extR \defeq [-\infty, \infty]$.
    \item[A convex function]
        A function $F: X \to \extR$ is convex if for all $x, x' \in X$ and $\lambda \in (0, 1)$, we have
        \[
            F(\lambda x + (1-\lambda)x') \le F(\lambda x)+ F((1-\lambda)x').
        \]
    \item[A concave function]
        A function $F: X \to \extR$ is concave if $-f$ is convex.
    \item[A convex-concave function]
        A function $K: X \times Y \to \extR$ is convex-concave if $K(\freevar, y)$ is convex for all $y \in Y$, and $K(x, \freevar)$ is concave for all $x \in X$.
    \item[The dual space]
        We write $X^*$ for the dual space of a topological vector (Banach, Hilbert) space $X$.
    \item[Set-valued map]
        We write $A: X \setto Y$ if $A$ is a set-valued map between the spaces $X$ and $Y$.
    \item[Derivative]
        We write $DF: X \to X^*$ for the derivative of a Gâteaux-differentiable function $F: X \to \R$.
    \item[Convex subdifferential]
        This is the map $\subdiff F: X \setto X^*$ for a convex $F: X \to \extR$. By definition $x^* \in \subdiff F(x)$ at $x \in X$ if and only if
        \[
            F(x')-F(x) \ge \dualprod{x^*}{x'-x}
            \quad(x' \in X).
        \]
    \item[Fenchel conjugate]
        This is the function $f^*: X^* \to \extR$ defined for $F: X \to \extR$ by
        \[
            f^*(x^*) \defeq \sup_{x \in X}  \dualprod{x^*}{x} - F(x)
            \quad(x^* \in X^*).
        \]
    \item[Fenchel preconjugate]
        If $X=(X_*)^*$ is the dual space of some space $X_*$, and $F: X \to \extR$, then $f_*: X_* \to \extR$ is the preconjugate of $f$ if $f=(f_*)^*$.
    \item[Proximal map]
        For a function $F: X \to \extR$, this can be defined as
        \[
            \prox_{F}(x) \defeq \argmin_{\alt x \in X}\left( F(\alt x)+\frac{1}{2}\norm{\alt x-x}_X^2\right).
        \]
        
    \item[Distributional derivative]
    It arises from integration by parts: If $u: \R^n \supset \Omega \to \R$ is differentiable and $\phi \in C_c^\infty(\Omega; \R^n)$, then
    \[
        \int_\Omega \iprod{\grad u}{\phi} \d x = - \int_\Omega u \divergence \phi \d x.
    \]
    If now $u$ is not differentiable, we \emph{define} the distribution $D \in C_c^\infty(\Omega; \R^n)^*$ by
    \[
        Du(\phi) \defeq - \int_\Omega u \divergence \phi \d x.
    \]
    If $Du$ is bounded (as a linear operator) it can be presented as a vector Radon measure \cite{federer1969gmt}, the space denoted $\Meas(\Omega; \R^n)$.

    \item[Indicator function] For a set $A$, we define
    \[
        \delta_A(x) \defeq \begin{cases} 0, & x \in A, \\ \infty, & x \not\in A. \end{cases}.
    \]
\end{list}

\bibliographystyle{jnsao}
 \providecommand{\eprint}[1]{\href{http://arxiv.org/abs/#1}{arXiv:#1}}
  \providecommand{\noopsort}[1]{}
  \providecommand{\eprint}[1]{\href{http://arxiv.org/abs/#1}{arXiv:#1}}



\end{document}